\documentclass[11pt,a4paper,reqno]{amsart}
\usepackage{vmargin,color}
\usepackage[latin1]{inputenc}
\usepackage{amsmath,amsfonts,amsthm,epsfig,graphicx}
\usepackage{caption}

\def\H{\mathcal{H}}

\def\U{\mathcal{U}}

\def\N{\mathbb N}

\def\R{\mathbb R}

\def\HH{\mathbf{H}}

\def\Om{\Omega}

\def\de{\delta}
\def\e{\varepsilon}
\def\k{\kappa}
\def\l{\lambda}

\def\om{\omega}
\def\vphi{\varphi}

\def\Lip{{\rm Lip}}

\def\Div{{\rm div}\,}

\def\Id{{\rm Id}\,}
\def\dist{{\rm dist}}

\def\spt{{\rm spt}}

\def\weakstar{\stackrel{*}{\rightharpoonup}}

\def\ov{\overline}
\def\pa{\partial}

\def\cc{\subset\subset}

\def\SS{\mathbb{S}}

\newtheorem{theorem}{Theorem}
\newtheorem{remark}{Remark}[section]

\newtheorem{lemma}[theorem]{Lemma}
\newtheorem{corollary}[theorem]{Corollary}

\def\var{\mathbf{var}}
\def\C{\mathbf{C}}
\def\D{\mathbf{D}}

\numberwithin{equation}{section}
\numberwithin{figure}{section}

\title{Critical points \\ in the Euclidean isoperimetric problem}

\title{Alexandrov's theorem revisited}

\author[Delgadino]{M. G. Delgadino}
\address[Matias Gonzalo Delgadino]{
\newline \indent Department of Mathematics, Imperial College London,
\newline \indent  South Kensington Campus, London SW7 2AZ, UK.}
\email{m.delgadino@imperial.ac.uk}

\author[Maggi]{F. Maggi}
\address[Francesco Maggi]{
\newline \indent Department of Mathematics, University of Texas at Austin,
\newline \indent 2515 Speedway STOP C1200, 78712, Austin, Texas, USA}
\email{maggi@math.utexas.edu}

\begin{document}

\begin{abstract}
{\rm We show that {\it among sets of finite perimeter} balls are the only volume-constrained critical points of the perimeter functional.}
\end{abstract}

\maketitle

\section{Introduction} \subsection{Sets of finite perimeter and the isoperimetric problem} The Euclidean isoperimetric theorem is probably the most basic result in the Calculus of Variations. There are many different proofs of the isoperimetry of balls in different classes of competitors, thus motivating the question: Which is the natural competition class in which the isoperimetric theorem can be formulated? From the perspective of the modern Calculus of Variations, the answer is found by looking at the relaxation of the perimeter functional. Following the seminal work of De Giorgi \cite{DeGiorgiSOFP1,DeGiorgiSOFP2} we consider as particulary natural his formulation of the Euclidean isoperimetric problem in the class of sets of finite perimeter. The characterization of Euclidean balls as the only isoperimetric sets among sets of finite perimeter was achieved by De Giorgi in \cite{DeGiorgi58ISOP}. By using the compactness properties of sets of finite perimeter, De Giorgi shows the existence of global minimizers (isoperimetric sets). Next, he shows that distributional perimeter is decreased under Steiner symmetrization, thus deducing that Steiner symmetrization applied to an isoperimetric set leads to an equality case in the Steiner perimeter inequality. He finally derives some necessary conditions for being an equality case in the Steiner perimeter inequality, so to deduce the sphericity of isoperimetric sets.

Despite the intimate connection between sets of finite perimeter and the isoperimetric problem, a characterization of balls as the only {\it critical points} in the isoperimetric problem {\it among sets of finite perimeter} is currently missing. The main result of this paper is showing the validity of this characterization.

The problem is already subtle in the case of local minimizers. By a local minimizer we mean a set of finite perimeter which minimizes perimeter among variations compactly supported in a fixed neighborhood of its own boundary. In particular, local minimality does not allow for perimeter comparison with sets obtained by symmetrization, thus ruling out the use of De Giorgi's original argument. In Euclidean spaces of dimension less or equal to $7$ the problem can be settled by the means of the regularity theory for local perimeter minimizers. In fact, in these dimensions any local minimizer is a bounded smooth set with constant mean curvature. One can then combine the strong maximum principle with the geometric construction known as the moving planes method (Alexandrov's theorem \cite{alexandrov}) to deduce the sphericity of the boundary.
But this strategy fails in dimension $8$ or larger, as boundaries of local perimeter minimizers could have, in principle, singular points, where local graphicality fails \cite{simons}. Actually, it has been recently shown that local volume-constrained perimeter minimizers in non-convex perturbations of the unit ball may indeed have singularities \cite{stezum}.

The problem is open in every dimension for critical points, that is, sets of finite perimeter and finite volume such that the first variation of perimeter under volume-fixing flows vanishes. These sets have constant mean curvature in a very natural (distributional) sense. However, at variance with the case of local minimizers, there seems to be no obvious way, even in low dimensions, to exploit regularity theorems and the moving planes method to conclude their sphericity.

Here we approach this problem by combining regularity theorems and maximum principles with various geometric constructions inspired by the proof of Alexandrov's theorem by Montiel and Ros \cite{montielros}. We thus extend De Giorgi's isoperimetric theorem from the case of global minimizers to that of critical points in the isoperimetric problem.

\begin{theorem}
  \label{thm main}
  Among sets of finite perimeter and finite volume, finite unions of balls with equal radii are the unique critical points of the Euclidean isoperimetric problem.
\end{theorem}

\begin{remark}
  {\rm Theorem \ref{thm main} is stated in terms of finite unions of balls. By assuming indecomposability (the measure-theoretic analogous of connectedness) of our critical points, we can change ``finite unions of balls'' with ``a single ball''. However, it seems natural to consider finite unions of mutually tangent balls as genuinely distinct critical points of the perimeter functional. Indeed, as proved in \cite{ciraolomaggi2017,delgadinomaggimihailaneumayer} (and as it has been known for a much longer time in the case of parameterized surfaces \cite{breziscoron84,struwe84}), finite unions of mutually tangent balls are the unique limits of sequences of bounded connected smooth sets with bounded perimeters and scalar mean curvatures which converge to a constant. In short, finite union of mutually tangent balls are the limit points of Palais-Smale sequences for the isoperimetric problem among connected open sets with smooth boundary.}
\end{remark}

\begin{remark}\label{remark guido}
  {\rm Wente's torus \cite{wentetorus} provides an example of an integer rectifiable varifold with multiplicity one in $\R^3$ which has constant distributional mean curvature and is not a sphere. Clearly, Wente's torus is not the boundary of a set of finite perimeter. From this point of view, Theorem \ref{thm main} seems to identify the most general family of surfaces such that constant distributional mean curvature implies sphericity.}
\end{remark}

While uniqueness and symmetry results for global minimizers can be obtained by a wealth of methods (symmetrization, mass transportation, etc), the methods employed in the case of critical points/solutions to geometric PDEs, that we are aware of, require a sufficient degree of smoothness (e.g. the classical Alexandrov theorem \cite{alexandrov}). Addressing this kind of issue without assuming smoothness seems a novel aspect of Theorem \ref{thm main}. This point could be particularly useful in proving convergence of geometric flows to union of balls. Indeed, without strong assumptions like convexity or star-shapedness, global-in-time existence results for geometric flows holds only in weak (either distributional or viscous) sense. Corollary \ref{corollary main} below should be useful in this context. To better illustrate this point, and to state the corollary itself, we introduce some terminology. In Theorem \ref{thm main} we  consider Borel sets $\Om$ in $\R^{n+1}$ with the following properties:

\medskip

\noindent {\it (i) Finite perimeter}: There exist a Borel set $\pa^*\Om$ which is covered, up to an $\H^n$-negligible set, by countably many graphs of $C^1$ functions from $\R^n$ to $\R^{n+1}$, and a Borel vector field $\nu_\Om:\pa^*\Om\to\SS^n$, such that
a generalized version of the divergence theorem holds
\begin{equation}
  \label{1}
  \int_\Om\,\Div X=\int_{\pa^*\Om}X\cdot\nu_\Om\,d\H^n\qquad\forall X\in C^1_c(\R^{n+1};\R^{n+1})\,.
\end{equation}
Here $\H^n$ denotes the $n$-dimensional Hausdorff measure on $\R^{n+1}$.

\medskip

\noindent {\it (ii) Constant distributional mean curvature}: There exists $\l\in\R$ such that
\begin{equation}
  \label{2}
  \int_{\pa^*\Om}\Div^{\pa^*\Om}X\,d\H^n=\l\,\int_{\pa^*\Om}X\cdot\nu_\Om\,d\H^n\qquad \forall X\in C^1_c(\R^{n+1};\R^{n+1})\,.
\end{equation}
Here $\Div^{\pa^*\Om}X=\Div X-\nu_\Om\cdot (\nabla X)[\nu_\Om]$ is the tangential divergence of $X$ along $\pa^*\Om$. Condition \eqref{2} is equivalent to asking that $\Om$ is a {\it critical point in the Euclidean isoperimetric problem}, that is
\begin{equation}
  \label{critical point x}
  \frac{d}{dt}\Big|_{t=0}P(f_t(\Om))=0
\end{equation}
whenever $\{f_t\}_{|t|<1}$ is a volume-preserving variation of $\Om$. Namely, each $f_t$ is a diffeomorphism with $f_t=\Id$ outside of a compact set, $f_0\equiv\Id$, and $|f_t(\Om)|=|\Om|$ for every $|t|<1$, where $|\Om|$ denotes the Lebesgue measure, or volume, of $\Om$. When $\Om$ is an open bounded set with $C^2$-boundary, as in Alexandrov theorem, one simply has $\pa^*\Om=\pa\Om$ and \eqref{critical point x} is equivalent to asking that $\pa\Om$ has constant mean curvature.

\medskip

With this terminology in place, we can state the following corollary of Theorem \ref{thm main}.

\begin{corollary}\label{corollary main}
  If $\{\Om_j\}_{j\in\N}$ and $\Om$ are sets of finite perimeter in $\R^{n+1}$ such that
  \begin{equation}
    \label{cor hp 1}
      \lim_{j\to\infty}|\Om_j\Delta\Om|=0\qquad\lim_{j\to\infty}P(\Om_j)=P(\Om)\,,
  \end{equation}
  and if the distributional mean curvatures of the $\Om_j$ converge to a constant $\l\in\R$, i.e.
  \begin{equation}
    \label{cor hp 2}
      \lim_{j\to\infty}\int_{\pa^*\Om_j}\big(\Div^{\pa^*\Om_j}\,X-\l\,X\cdot\nu_{\Om_j}\big)\,d\H^n=0\qquad \forall X\in C^1_c(\R^{n+1};\R^{n+1})\,,
  \end{equation}
  then $\l=n\,P(\Om)/(n+1)\,|\Om|$ and $\Om$ is a finite union of balls of radius $n/\l$.
\end{corollary}

\begin{remark}
  {\rm Notice that \eqref{cor hp 2} holds whenever each $\Om_j$ has distributional mean curvature $H_{\Om_j}\in L^p(\H^n\llcorner\pa^*\Om_j)$ for some $p\ge 1$ (see \eqref{HH def} and \eqref{HOmega} below) and
\begin{equation}
  \label{hp mcf}
  \lim_{j\to\infty}\int_{\pa^*\Om_j}|H_{\Om_j}-\l|^p\,d\H^n=0\,.
\end{equation}}
\end{remark}

\begin{remark}
  {\rm Global-in-time weak solutions of the volume-preserving mean curvature flow have been constructed in \cite{mugnaiseisspadaro} following the method proposed by Almgren, Taylor and Wang \cite{AlmgrenTW} and Luckhaus and Sturzenhecker \cite{luckhaussturzen}. Considering \cite[Theorem 2.3.2]{mugnaiseisspadaro} and Corollary \ref{cor hp 2}, it seems reasonable to conjecture that, for a large class of initial data and along time subsequences $t_j\to\infty$, the evolution $\{\Om(t):t\ge0\}$ should converge to finite union of balls. This is indeed the case, with a single ball as the limit for $t\to\infty$, when the initial data is uniformly smooth and convex, as proved in a classical theorem of Huisken \cite{huisken_vpmcf}. As geometric evolutions unavoidably produce singularities, Theorem \ref{thm main} could turn out to be a fundamental ingredient in attacking such questions.}
\end{remark}

\subsection{The Montiel-Ros argument}\label{sec Montiel Ros smooth argument} Our starting point is the beautiful proof of Alexandrov's theorem by Montiel and Ros \cite{montielros}, which we now recall. Assume that $\Om$ is a bounded open set with smooth boundary and positive mean curvature $H_\Om$ with respect to its outer unit normal $\nu_\Om$. Denote by $\{\k_i\}_{i=1}^n$ the principal curvatures of $\pa\Om$, indexed in increasing order so that $\k_n\ge H_\Om/n>0$, set $u(y)=\dist(y,\pa\Om)$ for each $y\in\Om$, and define
\begin{eqnarray}
  \label{psi smooth}
  &&Z=\Big\{(x,t)\in\pa\Om\times\R:0<t\le\frac1{\k_n(x)}\Big\}\,,
  \\\nonumber
  &&\zeta(x,t)=x-t\,\nu_\Om(x)\qquad (x,t)\in{Z}\,.
\end{eqnarray}
Let us denote by $B_\rho(x)$ the Euclidean ball in $\R^{n+1}$ with center at $x$ and radius $\rho$. If $y\in\Om$, then $B_{u(y)}(y)$ touches $\Om$ from inside at a point $x\in\pa\Om$, where $\k_n(x)\ge1/u(y)$, i.e. $u(y)\le1/\k_n(x)$. In particular, \begin{equation}
  \label{inclusion smooth}
  \Om\subset \zeta(Z)
\end{equation}
and by the area formula, with $J^{Z}\zeta$ denoting the tangential Jacobian of $\zeta$ along ${Z}$,
\begin{eqnarray*}
|\Om|&\le&|\zeta(Z)|\le\int_{\zeta({Z})}\H^0(\zeta^{-1}(y))\,dy=\int_{Z} J^{Z} \zeta\,d\H^{n+1}
\\
&=&\int_{\pa\Om}d\H^n_x\int_0^{1/\k_n(x)}\prod_{i=1}^n(1-t\,\k_i(x))\,dt\,.
\end{eqnarray*}
By the arithmetic-geometric mean inequality and by $\k_n\ge H_\Om/n$,
\begin{eqnarray}\label{same as}
|\Om|&\le&\int_{\pa\Om}d\H^n_x\int_0^{1/\k_n(x)}\Big(\frac1n\sum_{i=1}^n(1-t\,\k_i(x)\Big)^n\,dt
\\\nonumber
&\le&\int_{\pa\Om}d\H^n_x\int_0^{n/H_\Om(x)}\Big(1-t\,\frac{H_\Om(x)}n\Big)^n\,dt
=\frac{n}{n+1}\,\int_{\pa\Om}\frac{d\H^n}{H_\Om}\,,
\end{eqnarray}
so that we have proved the {\it Heintze-Karcher inequality}
\begin{equation}
  \label{hkinq}
  |\Om|\le\frac{n}{n+1}\,\int_{\pa\Om}\frac{d\H^n}{H_\Om}\,.
\end{equation}
If $H_\Om$ is constantly equal to some $\l\in\R$, then, by combining the divergence theorems \eqref{1} and \eqref{2} (see \eqref{H0 computation} below), we find $\l=n\,\H^n(\pa\Om)/(n+1)|\Om|$ . Hence equality holds throughout the argument, $\pa\Om$ is umbilical, and thus a sphere. In this way the Montiel-Ros argument provides a very effective proof of Alexandrov's theorem.

\subsection{The Montiel-Ros argument revisited} As the Montiel-Ros argument heavily relies on the smoothness of $\pa\Om$, it does not seem obvious how to adapt it to the case when $\Om$ is a set with finite volume, finite perimeter and constant distributional mean curvature.

From the point of view of regularity of $\pa\Om$, the starting point is given by the regularity theory of Allard \cite{Allard} (see \cite{SimonLN,DeLellisNOTES}). Up to modifying $\Om$ on a set of volume zero, we can assume that $\Om$ is open and that its topological boundary $\pa\Om$  can be split into a closed subset $\Sigma$ with $\H^n(\Sigma)=0$, and a relatively open subset $\pa^*\Om=\pa\Om\setminus\Sigma$ which is locally an analytic constant mean curvature hypersurface, characterized by the property that for every $x\in\pa\Om$
\[
x\in\pa^*\Om\qquad\mbox{if and only if}\qquad \lim_{\rho\to 0^+}\frac{\H^n(B_\rho(x)\cap\pa\Om)}{\rho^n}=\om_n\,,
\]
where $\om_n$ is the volume of the unit ball in $\R^n$. It is thus natural to redefine ${Z}$ by replacing $\pa\Om$ with $\pa^*\Om$ in \eqref{psi smooth}, i.e.
\begin{eqnarray}
  \label{psi non smooth}
  &&{Z}=\Big\{(x,t)\in\pa^*\Om\times\R:0<t\le\frac1{\k_n(x)}\Big\}\,,
\end{eqnarray}
where it is still true that the largest principal curvature $\k_n$ is positive along $\pa^*\Om$.

Given this choice of ${Z}$, in order to obtain \eqref{inclusion smooth} we would need to show that, for every $y\in\Om$, $B_{u(y)}(y)$ is touching $\pa\Om$ at a point $x\in\pa^*\Om$. This is not obvious as we just know that $\Sigma=\pa\Om\setminus\pa^*\Om$ is $\H^n$-negligible. Actually, this is false for an arbitrary point $y\in\Om$: this is the case when $\Om$ is a union of two mutually tangent balls, $x$ is a tangency point between two balls, and $y$ is any point between $x$ and the center of one of the balls. A cheap argument (see Lemma \ref{lemma cone stationary varifold}) show that at each touching point $x$, $\pa\Om$ blows-up an hyperplane {\it with integer multiplicity possibly larger than $1$}. So, near a touching point $x$, $\pa\Om$ consists of finitely many sheets that are mutually tangent at $x$. The union of these sheets has constant mean curvature in the distributional sense defined by \eqref{2}, although it is not immediate to extract information on the mean curvature each separate sheet. A deep result of Sch\"atzle \cite{schatzle} implies that the lower and upper sheets (with respect to any given direction) satisfy a measure-theoretic version of the strong maximum principle. This is a crucial information, which is delicate to exploit, but fundamental to our argument.

We now describe our argument by referring to the main steps of the proof of Theorem \ref{thm main}, which is contained in detail in section \ref{section criticla points}. We start by identifying a large subset $\Om^\star$ of good points of $\Om$, meaning that
\begin{equation}
  \label{inclusion non smooth}
  |\Om^\star\setminus\zeta({Z})|=0\,,\qquad |\Om\setminus\Om^\star|=0\,.
\end{equation}
In other words, the projection of almost every point in $\Om^\star$ onto $\pa\Om$ is contained in $\pa^*\Om$, and $\Om^\star$ is equivalent to $\Om$. The definition of $\Om^\star$ is as follows. First, for every $s>0$, we set
\begin{equation}
  \label{oms}
  \Om_s=\big\{y\in\Om:u(y)>s\big\} \qquad\pa\Om_s=\big\{y\in\Om:u(y)=s\big\}\,.
\end{equation}
Clearly $\Om_s$ satisfies an exterior ball condition of radius $s$ at each point of $\pa\Om_s$, but otherwise, $\Om_s$ is just a set of finite perimeter (for a.e. $s>0$). We can also obtain an interior ball condition, restricting ourselves to the following subset. Setting $t>s>0$, we define
\begin{equation}\label{gammst}
  \Gamma_s^t=\Big\{y\in\pa\Om_s:y=\Big(1-\frac{s}t\Big)\,x+\frac{s}{t}\,z\quad\mbox{for some $z\in\pa\Om_t$, $x\in\pa\Om$}\Big\}\,.
\end{equation}
Notice that $\Gamma_s^t$ is just a compact subset of $\pa\Om_s$, which could be very porous inside $\pa\Om_s$. Some technical effort (see step one)  is put in showing that $\Gamma_s^t$ can be covered by countably many $C^{1,1}$-images of $\R^n$ into $\R^{n+1}$, and that $\nabla u$ is tangentially differentiable along $\Gamma_s^t$ (with bounds on the tangential derivatives corresponding to the exterior/interior ball conditions).
Once these technical aspects are settled, we are allowed to use $\Id-r\,\nabla u$ to change variables between $\Gamma_{s}^t$ and $\Gamma_{s-r}^t$ and we can prove that $|\Om\setminus\Om^\star|=0$, where $\Om^\star$ is defined by
\begin{eqnarray}\label{omegastar}
\Gamma_s^+=\bigcup_{t>s}\Gamma_s^t\,,
\qquad
\Om^\star=\bigcup_{s>0}\Gamma_s^+\,.
\end{eqnarray}
This is done in step two of the proof.

Showing that $|\Om^\star\setminus\zeta({Z})|=0$, see step three and four, is considerably more delicate. We have to exclude that the points in a given $\Gamma_s^t$ that are projected into the singular set $\Sigma=\pa\Om\setminus\pa^*\Om$ have positive $\H^n$-measure, in other words, we want
\[
\H^n\big((\Id-s\,\nabla u)^{-1}(\Sigma)\cap\Gamma_s^t\big)=0\,.
\]
This may seem obvious, as $\Id-s\,\nabla u$ is almost injective on $\Gamma_s^t$ (see \eqref{h00}) and it is Lipschitz on each piece of a countable decomposition of $\Gamma_s^t$ (see \eqref{lip bound j}), while at the same time $\H^n(\Sigma)=0$. However we cannot derive a straightforward contradiction from the area formula, as the tangential Jacobian of $\Id-s\,\nabla u$ along $\Gamma_s^t$ may be zero $\H^n$-a.e. In fact, this is the information that we obtain from the area formula, namely, the least principal curvature of $\Gamma_s^t$ is equal to $-1/s$ along points in $(\Id-s\,\nabla u)^{-1}(\Sigma)\cap\Gamma_s^t$. Heuristically, this curvature for $\Gamma_s^t$ can only be obtained when $\partial\Omega$ has a inward corner, which is ruled out by absolute continuity of the mean curvature. Following this guiding example, we change variable to show that the least principal curvature of $\Gamma_{s-r}^t$ at corresponding points is thus as negative as we wish. This indicates that $\pa\Om_{s-r}$ has negative mean curvature on a set of positive $\H^n$-measure for any $r$ close enough to $s$. By the almost everywhere second order differentiability of $u$, swiping $r$ over an interval we can find a paraboloid with negative mean curvature, locally contained inside $\pa\Om_{s-r}$. By translating this object until it touches $\pa\Om$ (at $\Sigma$) we can apply Sch\"atzle maximum principle and derive a contradiction.

Having proved \eqref{inclusion non smooth} or \eqref{inclusion smooth}, we are ready to argue as Montiel and Ros. We thus find, from the equality case in their argument, that
\begin{eqnarray}
    \label{condition1 intro}
    \Big|\zeta({Z})\setminus\Om\Big|=0\,,&&
    \\
    \label{condition2 intro}
    \H^0(\zeta^{-1}(y))=1\,,&&\qquad\mbox{for a.e. $y\in\Om$}\,,
    \\
    \label{condition3 intro}
    \k_i(x)=\frac{H_\Om}n\,,&&\qquad\mbox{for every $x\in\pa^*\Om$, $i=1,...,n$}\,.
\end{eqnarray}
Condition \eqref{condition3 intro} implies that $\pa^*\Om$ is umbilical, in addition to being constant mean curvature. In particular, $\pa^*\Om$ consists of at most countably many open pieces of spheres with same curvature. Should these pieces be finitely many, one could conclude from the distributional constant mean curvature condition, in a rather direct way, that each piece is equal to a complete sphere. But as the number of the pieces could indeed be infinite, the pieces may have smaller and smaller areas and combine themselves in particular ways to achieve constant distributional mean curvature, creating at the same time a large singular set $\pa\Om\setminus\pa^*\Om$. To rule out this possibility, we exploit the information contained in \eqref{condition1 intro} and \eqref{condition2 intro} through a geometric argument.

\medskip

We conclude with two remarks. First, as a by-product of this analysis, we obtain an {\it Heintze-Karcher inequality for sets of finite perimeter} which are mean convex in a viscous sense, see Theorem \ref{thm viscous hk} below. This result is actually not needed to prove Theorem \ref{thm main}, but it is included as it may be considered of independent interest. Second, as recently shown by Brendle \cite{brendle}, the Montiel-Ros approach to Alexandrov's theorem is quite flexible, as it allows to show that constant mean curvature implies umbilicality in many warped product manifolds of physical and geometric interest. The methods of this paper should be naturally adaptable to these more general contexts. In this direction, in a companion paper \cite{delgadinomaggi_wulff_alex}, we prove that Wulff shapes are the only volume-constrained {\it local minimizers} of smooth uniformly elliptic surface tension energies. Of course the assumption of local minimality is considerably stronger than criticality.

\subsection{Organization of the paper} The paper is organized as follows. In section \ref{section gmt} we gather some background material from Geometric Measure Theory. In section \ref{section criticla points} we prove Theorem \ref{thm main} and Corollary \ref{corollary main}. The generalized Heintze-Karcher inequality for sets of finite perimeter is stated and proved in section \ref{section hk for sofp}.

\bigskip

\noindent{\bf Acknowledgments.} We thank Guido De Philippis and Massimiliano Morini for pointing out to us, respectively, the references \cite{wentetorus} and \cite{mugnaiseisspadaro}. Part of this work was completed while both authors were affiliated to the Abdus Salam International Centre for Theoretical Physics in Trieste, Italy. This work was supported by the NSF Grants DMS-1565354, DMS-1361122 and DMS-1262411. This work was also supported by the EPSRC under grant No. EP/P031587.

\section{Background material from Geometric Measure Theory}\label{section gmt} In this section we review some preliminaries from the theory of rectifiable sets (section \ref{section rectifiable sets}), rectifiable varifolds (section \ref{section varifold}) and sets of finite perimeter (section \ref{section sofp}). We refer to \cite{SimonLN,AFP,maggiBOOK,EvansGariepyBOOK} for detailed accounts. Finally, in section \ref{section critical points account}, we discuss some basic properties of volume-constrained critical points of the perimeter functional.

\subsection{Rectifiable sets}\label{section rectifiable sets} Denote by $\H^n$ the Hausdorff measure on $\R^{n+1}$. A Borel set $M\subset\R^{n+1}$ is a {\it locally $\H^n$-rectifiable set} if $M$ can be covered, up to a $\H^n$-negligible set, by countably many Lipschitz images of $\R^n$ into $\R^{n+1}$, and if $\H^n\llcorner M$ is locally finite on $\R^{n+1}$. We say that $M$ is {\it $\H^n$-rectifiable} if in addition $\H^n(M)<\infty$, and that $M$ is {\it normalized} if $M=\spt\,\H^n\llcorner M$, i.e.
\[
x\in M\qquad\mbox{if and only if}\qquad \H^n(B_\rho(x)\cap M)>0\qquad\forall\rho>0\,.
\]
Basic properties of rectifiable sets needed in the sequel are: (i) For $\H^n$-a.e. $x\in M$ there exists $T_xM\in G(n,n+1)$ (the space of $n$-dimensional planes in $\R^{n+1}$), such that
\begin{equation}
  \label{approximate tangent plane}
  \lim_{\rho\to 0^+}\int_{(M-x)/\rho}\vphi\,d\H^n=\int_{T_xM}\vphi\,d\H^n\qquad\forall\vphi\in C^0_c(\R^{n+1})\,,
\end{equation}
see \cite[Theorem 10.2]{maggiBOOK}. The plane $T_xM$ is called the {\it approximate tangent plane to $M$ at $x$}; (ii) If $M_1$ and $M_2$ are locally $\H^n$-rectifiable sets, then
\begin{equation}
  \label{locality tangent plane}
  T_x M_1=T_xM_2\qquad\mbox{$\H^n$-a.e. on $M_1\cap M_2$}\,,
\end{equation}
see \cite[Proposition 10.5]{maggiBOOK}; (iii) Lipschitz functions are differentiable along approximate tangent planes, that is, if $f:\R^{n+1}\to\R^{n+1}$ is a Lipschitz function, then, for $\H^n$-a.e. $x\in M$ such that $T_xM$ exists, the restriction of $f$ to $x+T_xM$ is differentiable at $x$, and the limit
\[
(\nabla^Mf)_x[\tau]=\lim_{h\to 0^+}\frac{f(x+h\tau)-f(x)}{h}\qquad\forall \tau\in T_xM\,,
\]
defines {\it tangential gradient $\nabla^Mf(x)=(\nabla^Mf)_x$ of $f$ along $M$ at $x$}; see \cite[Theorem 11.4]{maggiBOOK}; (iv) The tangential gradient just depends on the restriction of $f$ to $M$. In other words, if $f:M\to\R^{n+1}$ is a Lipschitz function, and $F,G:\R^{n+1}\to\R$ are Lipschitz functions such that $F=G=f$ on $M$, then
\begin{equation}
  \label{locality diff}
  \nabla^MF=\nabla^MG\qquad\mbox{$\H^n$-a.e. on $M$}\,.
\end{equation}
(v) Finally, given a Lipschitz function $f:M\to\R^{n+1}$, the {\it tangential Jacobian of $f$ along $M$} is defined at $\H^n$-a.e. $x\in M$ by
\[
J^Mf(x)=\sqrt{\det(\nabla^Mf(x)^*\nabla^Mf(x))}=\Big|\bigwedge_{i=1}^n (\nabla^Mf)_x[\tau_i(x)]\Big|
\]
provided $\{\tau_i(x)\}_{i=1}^n$ is an orthonormal basis of $T_xM$, and the area formula
\begin{equation}
  \label{area formula}
  \int_{f(M)}\H^0(f^{-1}(y))\,d\H^n_y=\int_M\,J^Mf(x)\,d\H^n_x
\end{equation}
holds \cite[Theorem 11.6]{maggiBOOK}.

For the lack of precise reference we justify property (iv). If $\psi:\R^n\to\R^{n+1}$ is a Lipschitz map and $E\subset\R^n$ is a Borel set, then by \cite[Lemma 10.4, Lemma 11.5]{maggiBOOK} we have $T_xM=(\nabla\psi)_{\psi^{-1}(x)}[\R^n]$ for $\H^n$-a.e. $x\in M\cap\psi(E)$, with
\begin{equation}
  \label{comp}
  (\nabla^MF)_x[\tau]=\nabla(F\circ\psi)_{\psi^{-1}(x)}\big[(\nabla\psi)^{-1}_{x}[\tau]\big]\qquad\forall \tau\in  T_xM\,.
\end{equation}
Since $F=G$ on $M$ implies $\nabla (F\circ\psi)=\nabla (G\circ \psi)$ $\H^n$-a.e. on $E\cap\psi^{-1}(M)$ \cite[Lemma 7.6]{maggiBOOK} we deduce \eqref{locality diff} from \eqref{comp}.

\subsection{Integer rectifiable varifolds}\label{section varifold} If $M$ is a $C^2$-hypersurface without boundary in $\R^{n+1}$, then the mean curvature vector $\HH_M\in C^0(M;\R^{n+1})$ of $M$ is such that
\begin{equation}
  \label{tan dt}
  \int_M\,\Div^M\,X\,d\H^n=\int_M\,\HH_M\cdot X\,d\H^n\,,\qquad\forall X\in C^1_c(\R^{n+1};\R^{n+1})\,,
\end{equation}
with $\HH_M(x)\cdot\tau=0$ for every $\tau\in T_xM$.
This basic fact motivates the following definitions.

Let $M$ be a locally $\H^n$-rectifiable set, and consider a Borel measurable function $\theta\in L^1_{{\rm loc}}(\H^n\llcorner M;\N)$. The {\it integer rectifiable varifold $\var(M,\theta)$} defined by $M$ and $\theta$, is the Radon measure on $\R^{n+1}\times G(n,n+1)$ defined as
\[
\int_{\R^{n+1}\times G(n,n+1)}\Phi\,d\var(M,\theta)=\int_M\,\Phi(x,T_xM)\,\theta(x)\,d\H^n_x\,,
\]
for every bounded, compactly supported Borel function $\Phi$ on $\R^{n+1}\times G(n,n+1)$. To each $X\in C^1_c(\R^{n+1};\R^{n+1})$ we associate the test function
\[
\Phi_X(x,T)=(\Div^T X)(x)\qquad (x,T)\in\R^{n+1}\times G(n,n+1)\,,
\]
where $\Div^TX$ is the divergence of $X$ with respect to $T$. Motivated by \eqref{tan dt}, we say that $\var(M,\theta)$ has {\it distributional mean curvature vector} $\HH_M\in L^1_{{\rm loc}}(\theta\,\H^n\llcorner M;\R^{n+1})$ if
\begin{equation}
  \label{HH def}
  \int_M\,\Div^M\,X\, \theta\,d\H^n=\int_M\,\HH_M\cdot X\,\theta\,d\H^n\,,\qquad\forall X\in C^1_c(\R^{n+1};\R^{n+1})\,.
\end{equation}
(The dependency of $\HH_M$ from $\theta$ is omitted.) When $|\HH_M|$ is constant ($\H^n$-a.e. on $M$) we say that $\var(M,\theta)$ has {\it constant distributional mean curvature on $\R^{n+1}$}; when $\HH_M=0$ we say that $\var(M,\theta)$ is {\it stationary on $\R^{n+1}$}. For example, if $M$ if a union of finitely many {\it possibly intersecting} spheres with same radius, then $M$ has constant distributional mean curvature in $\R^{n+1}$. Similarly, a finite union of hyperplanes is stationary in $\R^{n+1}$.

In the proof of Theorem \ref{thm main} we will exploit two forms of the maximum principle for integer rectifiable varifolds. The first one is a simple fact, well-known to experts, whose proof is included for the sake of clarity.

\begin{lemma}
  \label{lemma cone stationary varifold} Let $M$ be a normalized locally $\H^n$-rectifiable set such that $\var(M,\theta)$ is stationary on $\R^{n+1}$. If $M$ is a cone (that is, $M=t\,M$ for every $t>0$), and $M$ is contained in a closed half-space $H$ with $0\in\pa H$, then $M=\pa H$ and $\theta$ is constant. In particular, $M$ cannot be contained in the convex intersection of two distinct, non-opposite half-spaces containing the origin.
\end{lemma}

\begin{proof}
  Let $H=\{z\in\R^{n+1}:z\cdot\nu<0\}$ where $\nu\in\SS^n$. Given $\vphi\in C^\infty_c([0,\infty))$ with $0\le\vphi\le 1$, $\vphi(r)=1$ on $[0,\e)$ for some $\e>0$, and $\vphi'(r)<0$ on $\{0<\vphi<1\}$, let us set $X(x)=\vphi(|x|)\,\nu$ for $x\in\R^{n+1}$. Then $X\in C^\infty_c(\R^{n+1};\R^{n+1})$ and $\nabla X=\vphi'(|x|)\nu\otimes\hat{x}$, where $\hat{x}=x/|x|$ if $x\ne 0$. Let $\nu_M:M\to\SS^n$ be a Borel vector-field such that $T_xM=\nu_M(x)^\perp$ for $\H^n$-a.e. $x\in M$. Since $\hat{x}\cdot\nu_M(x)=0$ for $\H^n$-a.e. $x\in M$, we have
  \[
  \Div^{ M}X=\Div\,X-\nu_M\cdot\nabla X[\nu_M]=\vphi'(|x|)\,\Big(\nu\cdot\hat{x}-(\nu_M\cdot\nu)(\nu_M\cdot\hat x)\Big)=\vphi'(|x|)\,(\nu\cdot\hat{x})\,,
  \]
  and thus
  \[
  0=\int_M\Div^MX\,\theta\,d\H^n=\int_{ M}\vphi'(|x|)\,(\nu\cdot\hat{x})\,\theta(x)\,d\H^n(x)\,.
  \]
  Since $M\subset H$ implies $\hat x\cdot\nu\le0$ for every $x\in M$, $x\ne 0$, thanks to the arbitrariness of $\vphi$ we find $\nu\cdot\hat{x}=0$ for $\H^n$-a.e. $x\in M$. The lemma is proved.
\end{proof}

The second tool we shall use is a much deeper result, namely, Sch\"atzle's strong maximum principle for integer rectifiable varifolds with sufficiently summable distributional mean curvature \cite{schatzle}. The statement we adopt here is a slightly simplified version, still sufficient for our purposes, of \cite[Theorem 6.2]{schatzle}.

\begin{theorem}\label{thm mp}
  Let $M$ be a normalized locally $\H^n$-rectifiable set with distributional mean curvature vector $\HH_M\in L^p(\theta\,\H^n\llcorner M;\R^{n+1})$ for some $p>\max\{2,n\}$.

  Pick $\nu\in\SS^n$, $h_0\in\R$, and consider a connected open set $U\subset \nu^\perp$ such that
  \begin{equation}
    \label{def of f}
      \vphi(z)=\inf\Big\{h>h_0:z+h\,\nu\in  M\Big\}\qquad z\in U\,,
  \end{equation}
  satisfies $\vphi(z)\in(h_0,\infty)$ for every $z\in U$.

  If $\eta\in W^{2,p}(U;(h_0,\infty))$ is such that $\eta\le\vphi$ on $U$ and $\eta(z_0)=\vphi(z_0)$ for some $z_0\in U$, then it cannot be
  \begin{equation}
    \label{signs}
      -\Div\Big(\frac{\nabla \eta}{\sqrt{1+|\nabla \eta|^2}}\Big)(z)\le \HH_M(z+\vphi(z)\nu)\cdot\,\frac{-\nabla \vphi(z)+\nu}{\sqrt{1+|\nabla \vphi(z)|^2}}
  \end{equation}
  for $\H^n$-a.e. $z\in U$, unless $\eta=\vphi$ on $U$.
\end{theorem}

The signs in \eqref{signs} and the geometric intuition behind Theorem \ref{thm mp} are illustrated in
\begin{figure}
  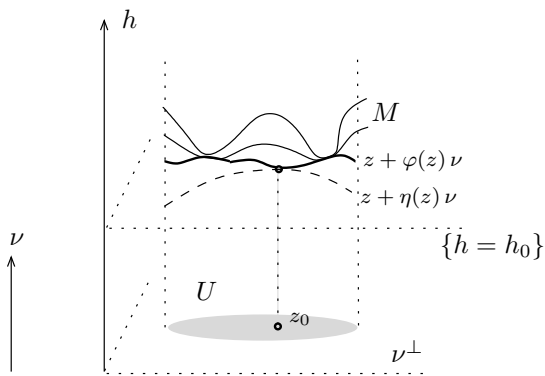\caption{{\small The strong maximum principle for integer varifolds. The rectifiable set $M$ may consist of multiple sheets which, combined together with the multiplicity function $\theta$, have distributional mean curvature $\HH_M$ in some $L^p$. The sheets may overlap in complicated ways along sets of positive area, so there is a nontrivial relation between the mean curvature vector $\HH_M$ of the whole configuration, and that of a single sheet. The function $\vphi$ describes the lower sheet of $M$ {\it above height $h_0$ with respect to the direction $\nu$} and projecting over an open set $U\subset\nu^\perp$. This lower sheet is shown to satisfy a strong maximum principle. Notice that the role of $h_0$ is that of localizing the part of the varifold we are looking at. For example, in this picture, $M$ could have many more points of the form $z+h\,\nu$ with $h<h_0$ and $z\in U$, but these points will not contribute to the definition of $\vphi$.}}\label{fig mp}
\end{figure}
Figure \ref{fig mp}.
The left-hand side is the mean curvature of the subgraph of $\eta$ with respect to its outer unit normal $(-\nabla \eta+\nu)/\sqrt{1+|\nabla \eta|^2}$, and, similarly, the right-hand side is the mean curvature of the subgraph of $\vphi$ with respect to its outer unit normal. So, if $\eta$ touches $\vphi$ from below at $z_0$, it cannot be that the subgraph of $\eta$ is in average bent upwards at least as much as the subgraph of $\eta$, unless $\eta=\vphi$. The considerable difficulty of the theorem lies in the fact that $\HH_M$ does not come into play as the mean curvature of the graph of $\vphi$, but rather as the mean curvature of a more complex structure (the integer rectifiable varifold $\var(M,\theta)$), of which $\vphi$ only represents a sort of lower envelope localized in the cylinder $\{z+t\nu:z\in U\,,t>h_0\}$.

\subsection{Sets of finite perimeter}\label{section sofp} A Borel set $\Om\subset\R^{n+1}$ has {\it locally finite perimeter} if there exists an $\R^{n+1}$-valued Radon measure $\mu_\Om$ on $\R^{n+1}$ such that
\begin{equation}
  \label{gauss green 0}
  \int_\Om\Div\,X=\int_{\R^{n+1}}X\cdot d\mu_\Om\qquad\forall X\in C^1_c(\R^{n+1};\R^{n+1})\,.
\end{equation}
The {\it perimeter of $\Om$ relative to an open set $A$} is defined as $P(\Om;A)=|\mu_\Om|(A)$, where $|\mu_\Om|$ is the total variation of $\mu_\Om$, and $\Om$ has {\it finite perimeter} if $P(\Om)=P(\Om;\R^{n+1})<\infty$. In this case, either $\Om$ or its complement has finite volume. By exploiting \eqref{gauss green 0}, the support of $\mu_\Om$ is seen to satisfy
\begin{equation}
  \label{spt mu Omega}
  \spt\mu_\Om=\Big\{x\in\R^{n+1}:0<|B_\rho(x)\cap\Om|<\om_n\rho^n\quad\forall\rho>0\Big\}\subset\pa\Om\,,
\end{equation}
see \cite[Proposition 12.19]{maggiBOOK}. Notice that $\spt\mu_\Om$ is invariant by zero-volume modifications of $\Om$, while of course $\pa\Om$ is not. The {\it reduced boundary} of a set of locally finite perimeter $\Om$ is defined as the set of points such that
\begin{equation}
  \label{measuretheoretic outer unit normal}
  \nu_\Om(x)=\lim_{\rho\to 0^+}\frac{\mu_\Om(B_\rho(x))}{|\mu_\Om|(B_\rho(x))}\quad\mbox{exists and belongs to $\SS^n$}\,.
\end{equation}
The Borel vector-field $\nu_\Om:\pa^*\Om\to\SS^n$ is called the {\it measure-theoretic outer unit normal to $\Om$}, and we always have
\begin{equation}
  \label{normalization top boundary}
  \ov{\pa^*\Om}=\spt\mu_\Om\,.
\end{equation}
Moreover by \cite[Theorem 15.9]{maggiBOOK}, the reduced boundary is locally $\H^n$-rectifiable, with
\[
\mu_\Om=\nu_\Om\,\H^n\llcorner\pa^*\Om\,,\qquad P(\Om;A)=\H^n(A\cap\pa^*\Om)
\]
for every open set $A\subset\R^{n+1}$, and thus \eqref{gauss green 0} takes the form
\begin{equation}
  \label{gauss-green}
  \int_\Om\,\Div X\,=\int_{\pa^*\Om}X\cdot\nu_\Om\,d\H^n\qquad\forall X\in C^1_c(\R^{n+1};\R^{n+1})\,,
\end{equation}
In addition, for every $x\in\pa^*\Om$, $\nu_\Om(x)^\perp=T_x(\pa^*\Om)$ is the approximate tangent plane to $\pa^*\Om$ at $x$ and in particular we have
\begin{equation}
  \label{red bndrt omn}
  \lim_{\rho\to0^+}\frac{\H^n(B_\rho(x)\cap\pa^*\Om)}{\rho^n}=\om_n\qquad\forall x\in\pa^*\Om\,.
\end{equation}
To every set $\Om$ of locally finite perimeter we can always associate in a natural way an integer rectifiable varifold $\var(\pa^*\Om,1)$. If $\var(\pa^*\Om,1)$ admits a distributional mean curvature vector $\HH_{\pa^*\Om}$, then the {\it distributional mean curvature of $\Om$} is defined by setting
\begin{equation}
  \label{HOmega}
  H_\Om=\HH_{\pa^*\Om}\cdot\nu_\Om\,.
\end{equation}
The subscript $\Om$ on $H_\Om$ reminds that we have used the outer orientation of $\Om$ to specify the scalar curvature. With this notation, $H_{B_r}=n/r$ for every $r>0$.

\subsection{Basic properties of critical points}\label{section critical points account} Here we prove some properties of critical points in the isoperimetric problem which descend from generally known facts about integer varifolds and sets of finite perimeter. A set of finite perimeter and finite volume $\Om$ is a {\it critical point for the isoperimetric problem} if
\begin{equation}
  \label{critical point}
  \frac{d}{dt}\Big|_{t=0}P(f_t(\Om))=0
\end{equation}
whenever $\{f_t\}_{|t|<1}$ is a one-parameter family of diffeomorphisms with $f_0=\Id$, $|f_t(\Om)|=|\Om|$ and $\spt (f_t-\Id)\cc\R^{n+1}$ for every $|t|<1$. By \cite[Theorem 17.20]{maggiBOOK}, \eqref{critical point} is equivalent to the existence of a constant $\l\in\R$ such that
\begin{equation}
  \label{cmc}
  \int_{\pa^*\Om}\Div^{\pa^*\Om}X\,d\H^n=\l\,\int_{\pa^*\Om}X\cdot\nu_\Om\,d\H^n\,,\qquad\forall X\in C^1_c(\R^{n+1};\R^{n+1})\,.
\end{equation}

\begin{lemma}\label{lemma prep}
  If $\Om\subset\R^{n+1}$ is a critical point for the isoperimetric problem, then $\Om$ is (equivalent modulo sets of volume zero to) a bounded open set such that $\pa\Om=\spt\mu_\Om$ and $\H^n(\pa\Om\setminus\pa^*\Om)=0$. Moreover, the constant $\l$ in \eqref{cmc} is equal to
  \begin{equation}
  \label{lagrange}
  H^0_\Om=\frac{n\,P(\Om)}{(n+1)|\Om|}\,,
  \end{equation}
  that is, $H_\Om\equiv H^0_\Om$. Finally,
  \[
  \pa^*\Om=\Big\{x\in\pa\Om:\lim_{\rho\to 0^+}\frac{\H^n(B_\rho(x)\cap\pa\Om)}{\rho^n}=\om_n\Big\}
  \]
  is locally an analytic hypersurface with constant mean curvature, relatively open in $\pa\Om$.
\end{lemma}

\begin{proof} By \cite[Theorem 17.6]{SimonLN}, condition \eqref{cmc} implies that for every $x\in\R^{n+1}$,
\begin{equation}
  \label{mono}
  e^{|\l|\rho}\frac{\H^n(B_\rho(x)\cap\pa^*\Om)}{\rho^n}\quad\mbox{is increasing on $\rho>0$}\,,
\end{equation}
which combined with \eqref{red bndrt omn} and \eqref{normalization top boundary} gives
\begin{equation}
  \label{uniform lb}
  \H^n(B_\rho(x)\cap\pa^*\Om)\ge \om_n\,e^{-|\l|}\,\rho^n\qquad\forall \rho\in(0,1),x\in\spt\mu_\Om\,.
\end{equation}
A first consequence of the lower bound \eqref{uniform lb} is that
\begin{equation}
  \label{equivalence in area}
  \H^n(\spt\mu_\Om\setminus\pa^*\Om)=0\,,
\end{equation}
see, e.g., \cite[Exercise 17.19]{maggiBOOK}; moreover, by combining \eqref{uniform lb} with $P(\Om)<\infty$ and a covering argument, we see that $\spt\mu_\Om$ is bounded.

Let us now consider the open set $\Om_1$ of those $x\in\R^{n+1}$ such that $|\Om\cap B_\rho(x)|=|B_\rho(x)|$ for every $\rho$ small enough, and the open set $\Om_0$ of those $x\in\R^{n+1}$ such that $|\Om\cap B_\rho(x)|=0$ for every $\rho$ small enough, so that
\begin{equation}
  \label{thanks to}
  \spt\mu_\Om=\R^{n+1}\setminus(\Om_0\cup\Om_1)\,,
\end{equation}
thanks to \eqref{spt mu Omega}. If $\Om^{(1)}$ denotes the set of points of density $1$ of $\Om$, then $\Om_1\subset\Om^{(1)}$, while
\[
|\Om^{(1)}\setminus\Om_1|=|\Om^{(1)}\cap\Om_0|+|\Om^{(1)}\cap\spt\mu_\Om|=|\Om^{(1)}\cap\spt\mu_\Om|=0
\]
as $\H^n(\spt\mu_\Om)<\infty$ thanks to \eqref{equivalence in area}. Thus $|\Om^{(1)}\Delta\Om_1|=0$, and then $|\Om\Delta\Om_1|=0$ by the Lebesgue density theorem. Since $\Om_0$ and $\Om_1$ are disjoint open sets, \eqref{thanks to} implies $\pa\Om_1\subset\spt\mu_\Om$. At the same time, $|\Om\Delta\Om_1|=0$ and the inclusion in \eqref{spt mu Omega} imply $\spt\mu_\Om\subset\pa\Om_1$. Hence $\spt\mu_\Om=\pa\Om_1$, and since $\spt\mu_\Om=\pa\Om_1$ is bounded and $|\Om_1|<\infty$, we have that $\Om_1$ is bounded. The first part of the statement is proved.

We show that $\l$ in \eqref{cmc} satisfies $\l=H^0_\Om$ with $H^0_\Om$ defined in \eqref{lagrange}. Since $\Om$ is bounded we can test both \eqref{gauss-green} and \eqref{cmc} with $X\in C^1_c(\R^{n+1};\R^{n+1})$ with $X(x)=x$ for $x$ in a neighborhood of $\Om$. Hence,
\begin{eqnarray}\nonumber
  (n+1)|\Om|&=&\int_\Om\Div(x)\,dx=\int_\Om\Div X=\int_{\pa^*\Om}X\cdot\nu_\Om\,d\H^n=\frac1\l\int_{\pa^*\Om}\Div^{\pa^*\Om}X\,d\H^n
  \\\label{H0 computation}
  &=&\frac1\l\,\int_{\pa^*\Om}\Div^{\pa^*\Om}(x)\,d\H^n_x=\frac{n\,P(\Om)}\l\,,
\end{eqnarray}
and thus $\l=H^0_\Om$.

Finally, by applying Allard's regularity theorem (see \cite[Theorem 24.2]{SimonLN} or \cite{DeLellisNOTES}) to $\var(\pa\Om,1)$, we see that $\pa\Om$ is an analytic constant mean curvature hypersurface in a neighborhood of every $x\in\pa\Om$ such that
\begin{equation}
  \label{reg}
  \lim_{\rho\to 0^+}\frac{\H^n(B_\rho(x)\cap\pa\Om)}{\rho^n}=\om_n\,.
\end{equation}
In particular, if  $x\in\pa\Om$ satisfies \eqref{reg} then there exists $\rho>0$ such that $B_\rho(x)\cap\Om$ is the epigraph of an analytic function, and thus $x\in\pa^*\Om$. Viceversa, \eqref{reg} holds everywhere on $\pa^*\Om$ thanks to \eqref{red bndrt omn}.
\end{proof}

We also notice a simple consequence of Lemma \ref{lemma cone stationary varifold}.

\begin{lemma}\label{lemma negligible opposites}
  If $\Om\subset\R^{n+1}$ is a critical point for the isoperimetric problem, $x\in\pa\Om$, and $y_1, y_2\in\Om$ are such that $|y_i-x|=\dist(y_i,\pa\Om)$ and $|x-y_1|=|x-y_2|$, then $x-y_1=y_2-x$.
\end{lemma}

\begin{proof}
  Since $\var(\pa\Om,1)$ is an integer varifold of constant distributional mean curvature, it admits at least one blow-up limit in the weak convergence of varifolds at $x$, and each such limit varifold is stationary and supported on a cone $M$; see \cite[Chapter 46]{SimonLN}. By construction, $M$ is contained in the half-spaces $\{z\cdot\nu_i\le 0\}$ defined by $\nu_i=(x-y_i)/|x-y_i|$, $i=1,2$. If $y_1\ne y_2$, then $\nu_1\ne\nu_2$, and Lemma \ref{lemma cone stationary varifold} implies that $\nu_1=-\nu_2$.
\end{proof}

\section{Critical points of the isoperimetric problem}\label{section criticla points} Referring to the introduction for the general strategy, we now present the proof of Theorem \ref{thm main}. At the end of the section we also prove Corollary \ref{corollary main}.

\begin{proof}[Proof of Theorem \ref{thm main}] Let $\Om$ be a set with finite perimeter and finite volume which is a critical point for the isoperimetric problem. The conclusion of Lemma \ref{lemma prep} is the starting point of our analysis, aimed at showing that $\Om$ is a finite union of disjoint balls of radius $n/H^0_\Om$. We rescale $\Om$ so that $H^0_\Om=n$.

\bigskip

\noindent {\it Properties of the distance function}: We set $u(y)=\dist(y,\pa\Om)$ for $y\in\R^{n+1}$, so that
\begin{equation}
  \label{Ny}
  N(y)=\nabla u(y)\in\SS^n\qquad\mbox{exists for a.e. $y\in\Om$}\,,
\end{equation}
thanks to Rademacher theorem. For $s>0$ we set
\[
\Om_s=\big\{y\in\Om:u(y)>s\big\}\qquad \pa\Om_s=\big\{y\in\Om:u(y)=s\big\}\,,
\]
and recall that, by the coarea formula \cite[Theorem 13.1,Theorem 18.1]{maggiBOOK}, $\Om_s$ is a set of finite perimeter for a.e. $s>0$, and for every Borel set $E\subset\R^{n+1}$,
\begin{equation}
  \label{coarea}
  |E|=\int_0^\infty\H^n(E\cap\pa^*\Om_s)\,ds=\int_0^\infty\H^n(E\cap\pa\Om_s)\,ds\,.
\end{equation}
In particular,
\begin{equation}
  \label{omegas meno omegas}
  \H^n(\pa\Om_s\setminus\pa^*\Om_s)=0\,,\qquad\mbox{for a.e. $s>0$}\,.
\end{equation}
We recall that for a.e. $y\in\Om$, $u$ admits a second order Taylor expansion at $y$. Indeed, given $A\subset\Om$ and $y\in \Om$, denote by $\bar\Theta(u,A)(y)$ the infimum of the constants $c>0$ such that for $a,b\in\R$ we have
\[
a+b\cdot z+c\,\frac{|z|^2}2\ge u(z)\qquad\forall z\in A\,,
\]
with equality at $y$. For any $y\in\Om$ we can pick $x\in\pa\Om$ such that $|x-y|=u(y)$,
\begin{equation}
  \label{bookpeople}
  u(z)=\dist(z,\pa\Om)\le\dist(z,\{x\})=|z-x|\,,\qquad\forall z\in\Om\,,
\end{equation}
that is, $z\mapsto|z-x|$ touches $u$ from above at $y$ over $\Om$. At the same time we can construct a second order polynomial that touches $z\mapsto|z-x|$ from above at $y$ over $\R^{n+1}$. Indeed, it holds
\begin{equation}
  \label{par up}
  |z-x|\le|y-x|+\frac{y-x}{|y-x|}\cdot(z-y)+\,\frac{|z-y|^2}{2|y-x|}\qquad\forall z\in\R^{n+1}\,.
\end{equation}
To check this set $y=x+t\,v$ for $t>0$ and $|v|=1$, and set $w=z-y$, so that \eqref{par up} becomes
\[
|t\,v+w|\le t+v\cdot w+\frac{|w|^2}{2t}\qquad\forall w\in\R^{n+1}\,.
\]
Taking squares this is equivalent to
\begin{eqnarray*}
t^2+2\,t\,v\cdot w+|w|^2&\le&t^2+2t\,v\cdot w+|w|^2+(v\cdot w)^2+\frac{(v\cdot w)\,|w|^2}{t}+\frac{|w|^4}{4t^2}
\\
&=&  t^2+2t\,v\cdot w+|w|^2+\Big(v\cdot w+\frac{|w|^2}{2t}\Big)^2\,,
\end{eqnarray*}
which clearly holds for every $w\in\R^{n+1}$. Thanks to \eqref{par up} there exists $a,b\in\R$ such that
\begin{eqnarray*}
|z-x|\le a+b\cdot z+\frac{|z|^2}{2|y-x|}\qquad\forall z\in\R^{n+1}\,,
\end{eqnarray*}
with equality if $z=y$, so that, by definition of $\bar\Theta$ and by \eqref{bookpeople}
\begin{equation}
  \label{cc}\bar\Theta(u,\Om)(y)\le \frac1{u(y)}\,,\qquad\forall y\in\Om\,.
\end{equation}
Arguing as in \cite[Proposition 1.6]{caffacabre}, we see that $u$ is twice differentiable a.e. in $\Omega$.

\bigskip

\noindent {\it Preliminary properties of the sets $\Gamma_s^t$}: For every $t>s>0$, we consider the compact set
\begin{equation}\label{gammst0}
  \Gamma_s^t=\Big\{y\in\pa\Om_s:y=\Big(1-\frac{s}t\Big)x+\frac{s}t\,z\quad\mbox{for some $z\in\pa\Om_t$, $x\in\pa\Om$}\Big\}\,.
\end{equation}
By definition, if $y\in\Gamma_s^t$, then there exist $x\in\pa\Om$ and $z\in\pa\Om_t$ such that
\begin{equation}
  \label{usala}
  B_{t-s}(z)\subset\Om_s\subset \R^{n+1}\setminus B_s(x)\,,\qquad \{y\}=\pa B_{t-s}(z)\cap\pa B_s(x)\,.
\end{equation}
In particular $x$ and $z$ are uniquely determined by the uniqueness of limits in $L^1_{{\rm loc}}$. Indeed, when $\rho\to0^+$,
\begin{equation}
  \label{omsyrho}
  \frac{\Om_s-y}{\rho}\to [x-z]^-\qquad\mbox{as characteristic functions in $L^1_{{\rm loc}}(\R^{n+1})$}
\end{equation}
where $[v]^-$ denotes the negative half-space defined by $v\ne 0$
\[
[v]^-=\Big\{w\in\R^{n+1}:w\cdot v<0\Big\}\,.
\]
Notice also that $\Lip(u;\R^{n+1})\le 1$ and the inclusion $B_{s+\e}(y-\e(x-z)/|x-z|)\subset\Om$ (which holds for $\e>0$ small since $t>s$) imply that $y$ has a unique projection onto the $\pa\Omega$. This shows that $u$ is differentiable at $y\in\Gamma_s^t$ with
\begin{equation}
  \label{N on Gammast}
  N(y)=-\frac{x-z}{|x-z|}\qquad\forall y=\Big(1-\frac{s}t\Big)x+\frac{s}t\,z\in\Gamma_s^t\,.
\end{equation}
In turn, \eqref{N on Gammast} gives
\begin{equation}
  \label{important}
  y+r\,N(y)\in\pa\Om_{s-r}\qquad\forall r\in[-s,t-s]\,,y\in\Gamma_s^t\,,
\end{equation}
By \eqref{important}, if $y,y'\in\Gamma_s^t$ then
\begin{eqnarray*}
s^2&\le&|y-s\,N(y)-y'|^2=s^2-2\,s\,N(y)\cdot(y-y')+|y-y'|^2\,,
\\
(t-s)^2&\le&|y+(t-s)\,N(y)-y'|^2=(t-s)^2+2\,(t-s)\,N(y)\cdot(y-y')+|y-y'|^2\,,
\end{eqnarray*}
that is
\begin{eqnarray}\label{C1 ext}
  |N(y)\cdot(y-y')|\le \max\Big\{\frac1s,\frac1{t-s}\Big\}\frac{|y-y'|^2}2\,,\qquad\forall y,y'\in\Gamma_s^t\,.
\end{eqnarray}
Using \eqref{N on Gammast} we easily see that $N$ is continuous on $\Gamma_s^t$, so that $(u,N)\in C^0(\Gamma_s^t;\R\times\R^{n+1})$ and satisfies \eqref{C1 ext}. By Whitney's extension theorem,  there exists $\phi\in C^1(\R^{n+1})$ such that $(\phi,\nabla\phi)=(u,N)$ on $\Gamma_s^t$. In particular, this implies the $\H^n$-rectifiability of $\Gamma_s^t$.

\bigskip

\noindent {\it Decomposition of $\Om$ and covering by $\zeta(Z)$}: We define
\begin{eqnarray}\label{omegastar0}
\Gamma_s^+=\bigcup_{t>s}\Gamma_s^t\,,
\qquad
\Om^\star=\bigcup_{s>0}\Gamma_s^+\subset\Om\,,
\qquad
Z=\Big\{(x,t)\in\pa^*\Om\times\R:0<t\le\frac1{\k_n(x)}\Big\}\,,
\end{eqnarray}
and set $\zeta(x,t)=x-t\,\nu_\Om(x)$. We claim that
\begin{equation}
  \label{omegastar proof}
 |\Om\setminus\Om^\star|=0\qquad |\Om^\star\setminus\zeta(Z)|=0\,.
\end{equation}
We divide the proof of \eqref{omegastar proof} into four steps.

\bigskip

\noindent {\it Step one}: We prove that $N$ is tangentially differentiable along $\Gamma_s^t$ at $\H^n$-a.e. $y\in\Gamma_s^t$, with
\begin{equation}
  \label{kst}
  \left\{
  \begin{split}
      &\nabla^{\Gamma_s^t}N(y)=-\sum_{i=1}^n(\k_s^t)_i(y)\,\tau_i(y)\otimes \tau_i(y)
      \\
      &-\frac1s\le(\k_s^t)_i(y)\le(\k_s^t)_{i+1}(y)\le\frac1{t-s}\,,
  \end{split}
  \right .
\end{equation}
where $\{\tau_i(y)\}_{i=1}^n$ is an orthonormal basis of $T_y\Gamma_s^t$. To this end, we first prove that $\Gamma_s^t$ can be covered by compact sets $\{\U_j\}_{j\in\N}$ in such a way that the restriction of $N$ to $\U_j$ is a Lipschitz map, that is
\begin{equation}
  \label{lip bound j}
  |N(y_1)-N(y_2)|\le C_j\,|y_1-y_2|\qquad\forall y_1,y_2\in\U_j\,.
\end{equation}
(In passing we notice that \eqref{lip bound j} implies the $C^{1,1}$-rectifiability of $\Gamma_s^+$, that is to say, the possibility of covering $\Gamma_s^+$ by graphs of $C^{1,1}$ functions from $\R^n$ to $\R^{n+1}$).

We start by defining the sets $\U_j$. Let us denote by
\[
\C(N,\rho)=\Big\{z+h\,N:z\in N^\perp\,,|z|<\rho\,,|h|<\rho\Big\}\,,
\]
the open cylinder centered at the origin with axis along $N\in\SS^n$, radius $\rho>0$, and height $2\rho$. Notice that, by the interior/exterior ball condition, $\Gamma_s^t$ admits an approximate tangent plane at $\H^n$-a.e. of its points, and this plane is then necessarily equal to $N(y)^\perp$, that is
\[
T_y\Gamma_s^t=N(y)^\perp\qquad\mbox{for $\H^n$-a.e. $y\in\Gamma_s^t$}\,.
\]
In particular \eqref{approximate tangent plane} implies
\[
\lim_{\rho\to 0^+}\frac{\H^n(\Gamma_s^t\cap\big(y+\C(N(y),\rho)\big))}{\rho^n}=\om_n\,,\qquad\mbox{for $\H^n$-a.e. $y\in\Gamma_s^t$}\,.
\]
By Egoroff's theorem, we can find compact sets $\U_j$ covering $\Gamma_s^t$ such that
\begin{equation}
  \label{mujstar}
  \mu_j^*(\rho)=\sup_{y\in\U_j}\Big|1-\frac{\H^n(\Gamma_s^t\cap\big(y+\C(N(y),\rho)\big))}{\om_n\,\rho^n}\Big|\to 0\qquad\mbox{as $\rho\to 0^+$}\,.
\end{equation}
Consider the function $\phi$ constructed in proving the $\H^n$-rectifiability of $\Gamma_s^t$. Since $\nabla\phi(y)=N(y)\ne 0$ at each $y\in\Gamma_s^t$, we can apply the implicit function theorem at $y$ and find that $\Gamma_s^t$ is a $C^1$-graph over a disk or radius $\rho_y$ in a neighborhood of $y$. We can thus pick any sequence $\rho_j\to 0^+$, and up to further subdivide $\U_j$ and relabel the resulting pieces, we can assume that each $\U_j$ has the following property: for each $y\in\U_j$ there exists
\begin{equation}
  \label{psij}
  \psi_j\in C^1(N(y)^\perp)\,,\quad
  \psi_j(0)=0\,,\quad \nabla\psi_j(0)=0\,,\quad \|\nabla\psi_j\|_{C^0(N(y)^\perp)}\le 1\,,
\end{equation}
such that, if
\begin{equation}
  \label{sigmaj'}
  \U_j'=\mbox{projection of $\U_j$ on $N(y)^\perp\cap\{|z|<\rho_j\}$},
\end{equation}
then
\begin{eqnarray}\label{sigma j graph}
\U_j\cap\Big(y+\C(N(y),\rho_j)\Big)&=&\Gamma_s^t\cap\Big(y+\C(N(y),\rho_j)\Big)
\\\nonumber
&=&y+\Big\{z+\psi_j(z)\,N(y):z\in\U_j'\Big\}\,.
\end{eqnarray}
(Notice that both $\psi_j$ and $\U_j'$ depend on the point $y\in\U_j$ at which we are considering the ``graphicality'' property of $\U_j$, but that this dependency is not stressed to simplify the notation.) If we set
\begin{equation}
  \label{mujrho}
  \mu_j(\rho)=\max\Big\{\mu_j^*(\rho),\max_{|z|\le \rho}|\nabla\psi_j(z)|\Big\}\qquad\rho\in(0,\rho_j]
\end{equation}
then $\mu_j(\rho)\to 0$ as $\rho\to 0^+$ by \eqref{mujstar} and by continuity of $\nabla\psi_j$. This completes the definition of the sets $\U_j$.

\medskip

We now prove \eqref{lip bound j}. Fix $y_1,y_2\in\U_j$. Let $\rho_j$ and $\psi_j$ be the functions associated to $\U_j$ and $y_2\in\U_j$ as we have just described. For $r_j<\rho_j/3$ to be chosen, we can directly assume that
\begin{equation}
  \label{y1}
  y_1\in y_2+\C(N(y_2),r_j)
\end{equation}
for otherwise $|y_1-y_2|\ge c(n)\,r_j$ and, trivially, $|N(y_1)-N(y_2)|\le 2\le C_j\,|y_1-y_2|$. Next we assume, as we can do without loss of generality up to a rigid motion, that
\[
y_2=(0,0)\in\R^n\times\R\,,\qquad N(y_2)=(0,1)\in\R^n\times\R\,,\qquad N(y_2)^\perp=\R^n\,.
\]
In this way \eqref{sigma j graph} takes the form
\begin{equation}\label{sigma j graph0}
\Big\{(z,h)\in\Gamma_s^t:|z|<\rho_j\,,|h|<\rho_j\Big\}=\Big\{(z,\psi_j(z)):z\in\U_j'\Big\}
\end{equation}
with
\begin{equation}
  \label{psijjj}
  \psi_j\in C^1(\R^n)\,,\quad \psi_j(0)=0\,,\quad \nabla\psi_j(0)=0\,,\quad \|\nabla\psi_j\|_{C^0(\R^n)}\le 1\,.
\end{equation}
By \eqref{y1}, $y_1=(z_1,\psi_j(z_1))$ for some $z_1\in\U_j'$ with $|z_1|< r_j$. By continuity of $N$ along $\Gamma_s^t$ and since $N(0)=(0,1)$ we find
\[
N(y_1)=\frac{(-\nabla\psi_j(z_1),1)}{\sqrt{1+|\nabla\psi_j(z_1)|^2}}\,.
\]
In particular,
\begin{eqnarray*}
  \frac{|N(y_1)-N(y_2)|^2}2=1-\frac1{\sqrt{1+|\nabla\psi_j(z_1)|^2}}
  \le\frac{|\nabla\psi_j(z_1)|^2}2\,,
\end{eqnarray*}
while at the same time $|y_1-y_2|^2=|z_1|^2+\psi_j(z_1)^2\ge|z_1|^2$. We are thus left to show
\begin{equation}
  \label{left to show}
  |\nabla\psi_j(z_1)|\le C_j\,|z_1|\,.
\end{equation}
To this end we would like to exploit \eqref{C1 ext} with $y=y_1$ and $y'=y_0$ where $y_0=(z_0,h_0)$ is defined, in terms of a suitable $e_0\in\SS^n$ (see \eqref{e0} below), as
\begin{equation}
  \label{z0t0}
  z_0=z_1-|z_1|\,e_0\qquad h_0=\psi_j(z_0)\,,
\end{equation}
Since $\Gamma_s^t$ may be very ``porous'', that is, its projection over $\{|z|<\rho_j\}$ could have lots of holes, it is not generally true that $y_0\in\Gamma_s^t$ and thus that $y'=y_0$ is an admissible choice in \eqref{C1 ext}. But when this is the case, by \eqref{C1 ext}
\begin{equation}
  \label{takes the form}
  C\,|y_1-y_0|^2\ge N(y_1)\cdot(y_1-y_0)=|z_1|\,\frac{\nabla\psi_j(z_1)\cdot (-e_0)}{\sqrt{1+|\nabla\psi_j(z_1)|^2}}
  +\frac{\psi_j(z_1)-\psi_j(z_0)}{\sqrt{1+|\nabla\psi_j(z_1)|^2}}\,.
\end{equation}
Now, in order to exploit \eqref{takes the form}, we notice that
\begin{equation}
  \label{notice}
  |\psi_j(z)|\le C\,|z|^2\qquad\mbox{$\forall |z|<\rho_j$ such that $(z,\psi_j(z))\in\Gamma_s^t$}\,,
\end{equation}
which is an immediate consequence of the fact that, around $(0,0)=(0,\psi_j(0))$, $\Gamma_s^t$ is trapped between two tangent balls (notice that we do not know this about the graph of $\psi_j$, and so we can apply \eqref{notice} only to the points of this graph that lie in $\Gamma_s^t$). Since $|z_0|\le 2\,|z_1|<2\,r_j<\rho_j$, still {\it assuming} that $y_0=(z_0,h_0)\in\Gamma_s^t$,  by \eqref{notice} we find that
\begin{eqnarray*}
|y_1-y_0|^2&=&|z_1|^2+(\psi_j(z_1)-\psi_j(z_0))^2\le C\,|z_1|^2
\\
\bigg|\frac{\psi_j(z_1)-\psi_j(z_0)}{\sqrt{1+|\nabla\psi_j(z_1)|^2}}\bigg|&\le&
|\psi_j(z_1)|+|\psi_j(z_0)|\le C\,|z_1|^2\,,
\end{eqnarray*}
and thus \eqref{takes the form} takes the form
\begin{equation}
  \label{takes the form 2}
  C\,|z_1|^2\ge |z_1|\,\frac{\nabla\psi_j(z_1)\cdot (-e_0)}{\sqrt{1+|\nabla\psi_j(z_1)|^2}}\,.
\end{equation}
Our choice of $e_0$ is thus clear, we want
\begin{equation}
  \label{e0}
  e_0=-\frac{\nabla\psi_j(z_1)}{|\nabla\psi_j(z_1)|}\,,
\end{equation}
to have a chance of proving \eqref{left to show}.

We are now ready to prove \eqref{left to show}. Set $y_0=(z_0,h_0)$ for $e_0$ as in \eqref{e0} and $z_0$ and $h_0$ as in \eqref{z0t0}. If $z_0\in\U_j'$, and thus $y_0\in\Gamma_s^t$, then, as explained, we are done. Otherwise, let $\e_0$ be the largest $\e>0$ such that
\[
\{|z-z_0|<\e\}\cap\U_j'=\emptyset\,.
\]
Since $z_1\in\U_j'$ and $|z_0-z_1|=|z_1|$, we have that $\e_0\le|z_1|$. In particular, since $|z_0|\le2|z_1|$, the ball $\{|z-z_0|<\e_0\}$ is contained in $\{|z|<3\,|z_1|\}\subset\{|z|<\rho_j\}$ thanks to $3r_j<\rho_j$. By definition of $\e_0$, there exists $z_*\in\U_j'$ with $|z_*-z_0|=\e_0$ and
\begin{eqnarray}\label{z0zstar}
\om_n\,|z_0-z_*|^n&=&\H^n\big(\{|z-z_0|<\e_0\}\big)
\\\nonumber
&\le&\H^n\big(\{|z|<3|z_1|\}\setminus\U_j'\big)
=\om_n\,(3|z_1|)^n-\H^n\big(\U_j'\cap \{|z|<3|z_1|\}\big)\,.
\end{eqnarray}
On the one hand, since $\U_j$ is the graph of the Lipschitz function $\psi_j$ over $\U_j'$
\begin{eqnarray*}
  \H^n\big(\U_j'\cap \{|z|<3|z_1|\}\big)&\le& \int_{\U_j'\cap \{|z|<3|z_1|\}}
  \sqrt{1+|\nabla \psi_j|^2}=\H^n\big(\U_j\cap\C(N(y_2),3|z_1|)\big)
  \\
  &=&\H^n\big(\Gamma_s^t\cap\C(N(y_2),3|z_1|)\big)
  \\
  &\le&\om_n\,(3|z_1|)^n\Big(1+\mu_j(3\,|z_1|)\Big)
\end{eqnarray*}
thanks to \eqref{mujrho}; on the other hand, again by the definition \eqref{mujrho} of $\mu_j$,
\begin{eqnarray*}
  \H^n\big(\U_j'\cap \{|z|<3|z_1|\}\big)&=&\int_{\U_j'\cap \{|z|<3|z_1|\}}\frac{\sqrt{1+|\nabla \psi_j|^2}}{\sqrt{1+|\nabla \psi_j|^2}}
  \\
  &\ge&\frac{\H^n\big(\Gamma_s^t\cap\C(N(y_2),3|z_1|)\big)}{\sqrt{1+\mu_h(3\,|z_1|)^2}}
  \\
  &\ge&\frac{1-\mu_j(3\,|z_1|)}{\sqrt{1+\mu_h(3\,|z_1|)^2}}\,\om_n\,(3|z_1|)^n\,.
\end{eqnarray*}
Combining the last two estimates into \eqref{z0zstar} we find
\[
\om_n\,|z_0-z_*|^n\le C\,\mu_j(3|z_1|)\,\om_n\,(3|z_1|)^n\,,
\]
that is
\begin{equation}
  \label{superfico}
  |z_0-z_*|\le C\,\mu_j(3\,|z_1|)^{1/n}\,|z_1|\,.
\end{equation}
In other words, after scaling out $|z_1|$, the best point we can use, $z_*$, is as close as we want to the point we would like to use $z_0$. We conclude the argument setting $y_*=(z_*,\psi_j(z_*))$. Since $z_*\in\U_j'$, we have $y_*\in\Gamma_s^t$. We can apply \eqref{C1 ext} with $y=y_1=(z_1,\psi_j(z_1))$ and $y'=y_*$, to find
\begin{eqnarray}\nonumber
C\,|y_1-y_*|^2&\ge&N(y_1)\cdot (y_1-y_*)\ge \frac{(-\nabla\psi_j(z_1))\cdot(z_1-z_*)}{\sqrt{1+|\nabla\psi_j(z_1)|^2}}
+\frac{\psi_j(z_1)-\psi_j(z_*)}{\sqrt{1+|\nabla\psi_j(z_1)|^2}}
\\\nonumber
&\ge& \frac{(-\nabla\psi_j(z_1))\cdot(z_1-z_*)}{\sqrt{1+|\nabla\psi_j(z_1)|^2}}
-C\,\Big(|z_1|^2+|z_*|^2\Big)
\\\label{perfect}
&\ge& |\nabla\psi_j(z_1)|\Big(1-C\,\mu_j(3\,|z_1|)^{1/n}\Big)\,\frac{|z_1|}{C}
-C\,\Big(|z_1|^2+|z_*|^2\Big)
\end{eqnarray}
where we have first applied \eqref{notice} to $z_1$ and $z_*$, and then have decomposed $z_1-z_*$ as the sum of $z_1-z_0=e_0\,|z_1|$ and of $z_0-z_*$, have recalled the definition of $e_0$, and have used \eqref{superfico}. Similarly,
\begin{eqnarray*}
|y_1-y_*|&\le&|z_1-z_*|+|\psi_j(z_1)-\psi_j(z_*)|\le |z_1-z_0|+|z_0-z_*|+ C\,\Big(|z_1|^2+|z_*|^2\Big)
\\
&\le&C\,|z_1|\,,
\end{eqnarray*}
and thus \eqref{perfect} implies \eqref{left to show}. This concludes the proof of \eqref{lip bound j}. We now prove \eqref{kst}.

\medskip

As noticed in section \ref{section varifold}, since $N$ is a Lipschitz function on each $\U_j$, and since the $\U_j$ are covering $\Gamma_s^t$, we deduce that $N$ is tangentially differentiable along $\Gamma_s^t$, and that its tangential gradient along $\Gamma_s^t$ can be computed by looking at any Lipschitz extension of $N$ to $\R^{n+1}$. Moreover, by \eqref{locality tangent plane}, it is enough to work with $\U_j$ in place of $\Gamma_s^t$.

\medskip

To construct a convenient extension of $N$ we go back to the proof of the $\H^n$-rectifiability of $\Gamma_s^t$, and this time we construct $\phi\in C^{1,1}(\R^{n+1})$ such that $(u,N)=(\phi,\nabla\phi)$ on $\U_j$ by taking \eqref{C1 ext} and \eqref{lip bound j} into account. Then we can go back to the construction of the sets $\U_j$, and apply the $C^{1,1}$-implicit function theorem to deduce that for each $y\in\U_j$ there exists
\[
\psi_j\in C^{1,1}(N(y)^\perp)\,,
\]
satisfying \eqref{psij} and \eqref{sigma j graph}. In particular, we can consider the Lipschitz extension $N_*$ of $N$ from $\U_j\cap(y+\C(N(y),\rho_j))$ to $y+\C(N(y),\rho_j)$ given by
\[
N_*(y+z+h\,N(y))=\frac{-\nabla\psi_j(z)+N(y)}{\sqrt{1+|\nabla\psi_j(z)|^2}}\,,\qquad\forall z\in N(y)^\perp\,, |z|<\rho_j\,,|h|<\rho_j\,.
\]
Setting $\Psi_j(z)=y+z+\psi_j(z)\,N(y)$ for $|z|<\rho_j$, by \eqref{comp} we have that for $\H^n$-a.e. $y'\in\U_j$,
\[
\big(\nabla^{\U_j}N\big)_{y'}[\tau]=\nabla(N_*\circ\Psi_j)_{\Psi_j^{-1}(y')}\,[e]
\]
where $\tau\in T_{y'}\U_j$ and $e=(\nabla\Psi_j)_{\Psi_j^{-1}(y')}[\tau]\in\R^n$. When $\psi_j\in C^2(N(y)^\perp)$, a classical computation shows that
\[
\nabla(N_*\circ\Psi_j)_z\,[e]=A_j(\Psi_j(z))[\tau]
\]
where $A_j$ denote the second fundamental form to the graph of $\psi_j$, which is symmetric thanks to commutativity property of the second derivatives of $\psi$; and where the eigenvalues of $A_j$ are bounded from below by $-1/s$ and from above by $1/(t-s)$ thanks to $\U_j\subset\Gamma_s^t$. In our case the same computations holds for a.e. $|z|<\rho_j$ by the chain rule for Lipschitz functions, where the symmetry of $A_j$ is guaranteed by the fact that $\nabla^2\psi_j$ is both a distributional gradient and an a.e. classical differential of $\nabla\psi_j$. Finally, the a.e.-pointwise estimates on the eigenvalues are deduced a.e. on $\U_j'$ thanks to the fact that $\nabla^2\psi_j$ is an a.e. classical differential. This proves \eqref{kst}.

\bigskip

\noindent {\it Step two}: We claim that for every $t>s>0$ we have
\begin{equation}
  \label{proj pp}
  \H^n(\pa\Om_t)\le \big(t/s\big)^n\,\H^n(\Gamma_s^t)\,,
\end{equation}
and then use \eqref{proj pp} to prove
\begin{equation}\label{omega omega star 0}
|\Om\Delta\Om^\star|=0\,.
\end{equation}
Indeed, for $r\in[-s,t-s]$ let us consider the map
\begin{equation}
  \label{fr}
  f_r:\Gamma_s^t\to\pa\Om_{s+r}\qquad f_r(y)=y+r\,N(y)\qquad y\in\Gamma_s^t\,.
\end{equation}
The fact that $f_r(y)\in\pa\Om_{s+r}$ is immediate as every $y\in\Gamma_s^t$ has the form $y=(1-(s/t))x+(s/t)z$ for $x\in\pa\Om$, $z\in\pa\Om_t$. Notice that, again by definition of $\Gamma_s^t$, $f_{t-s}$ is surjective, that is $\pa\Om_t=f_{t-s}(\Gamma_s^t)$. Thus
\begin{eqnarray*}
  \H^n(\pa\Om_t)=\H^n(f_{t-s}(\Gamma_s^t))\le\int_{f_{t-s}(\Gamma_s^t)}\H^0(f_{t-s}^{-1}(z))\,d\H^n_z=\int_{\Gamma_s^t}\,J^{\Gamma_s^t}f_{t-s}\,d\H^n
\end{eqnarray*}
where by \eqref{kst}, and in particular by the lower bound on $(\k_s^t)_i$,
\[
J^{\Gamma_s^t}f_{t-s}=\prod_{i=1}^n\Big(1-(t-s)\,(\k_s^t)_i\Big)\le  \bigg(1+\frac{t-s}{s} \bigg)^n\qquad\mbox{$\H^n$-a.e. on $\Gamma_s^t$}\,.
\]
This proves \eqref{proj pp}. To prove \eqref{omega omega star 0}, we first apply the coarea formula \eqref{coarea} to find
\begin{equation}
  \label{elegant}
  |\Om\Delta\Om^\star|=\int_0^\infty\,\H^n\Big((\Om\Delta\Om^\star)\cap\pa\Om_s\Big)\,ds=
\int_0^\infty\,\H^n\big(\pa\Om_s\setminus\Gamma_s^+\big)\,ds\,,
\end{equation}
where $\Gamma_s^+\subset\pa\Om_s$. Again by the coarea formula, for a.e. $s>0$,
\[
\H^n(\pa\Om_s)=\lim_{\e\to0}\frac{|\Om_s|-|\Om_{s+\e}|}{\e}=\lim_{\e\to 0^+}\frac1\e\int_0^\e\H^n(\pa\Om_{s+r})\,dr\,.
\]
where by \eqref{proj pp}
\[
\frac1\e\int_0^\e\H^n(\pa\Om_{s+r})\,dr\le \frac1\e\int_0^\e \bigg(1+\frac{r}{s} \bigg)^n\H^n(\Gamma_s^{s+r})\,dr\le\,
\bigg(1+\frac{\epsilon}{s} \bigg)^n\H^n(\Gamma_s^+)\,.
\]
Since $\Gamma_s^+\subset\pa\Om_s$, this proves
\begin{equation}
  \label{cool}
  \H^n(\Gamma_s^+)=\H^n(\pa\Om_s)\qquad\mbox{for a.e. $s>0$}\,,
\end{equation}
which, combined with \eqref{elegant} gives in turn \eqref{omega omega star 0}.

\bigskip

\noindent {\it Step three}: For $r\in(0,s)$, let us consider the map
\[
g_r:\Gamma_s^+\to\Gamma_{s-r}^+\,,\qquad g_r(y)=y-r\,N(y)\qquad y\in\Gamma_s^+\,,
\]
which is (clearly) a bijection between $\Gamma_s^t$ and $\Gamma_{s-r}^t$ for each $t>0$. We claim that if $y$ is a point of tangential differentiability of $N$ along $\Gamma_s^t$, then $g_r(y)$ is a point of tangential differentiability of $N$ along $\Gamma_{s-r}^t$, and
\begin{equation}
  \label{formula curvatures}
  (\k_{s-r}^t)_i(g_r(y))=\frac{(\k_s^t)_i(y)}{1+r\,(\k_s^t)_i(y)}\,,\qquad\forall i=1,...,n\,.
\end{equation}
Indeed, it is easily seen that
\begin{equation}
  \label{important2}
  N(y)=N(g_r(y))=N(y-r\,N(y))\qquad \forall y\in\Gamma_s^t\,,
\end{equation}
so that if $y$ is a point of tangential differentiability of $N$ along $\Gamma_s^t$ and $\tau\in T_y\Gamma_s^t$, then $\tau\in T_{g_r(y)}\Gamma_s^t$ and
\[
(\nabla^{\Gamma_s^t}N)_y[\tau]=\big(\nabla^{\Gamma_{s-r}^t}N\big)_{g_r(y)}\Big[\tau-r\,(\nabla^{\Gamma_s^t}N)_y[\tau]\Big]\,.
\]
Plugging in $\tau=\tau_i(y)$ as in \eqref{kst} we find
\[
-(\k_s^t)_i(y)\tau_i(y)=\big(1+r\,(\k_s^t)_i(y)\big)\,\big(\nabla^{\Gamma_{s-r}^t}N\big)_{g_r(y)}[\tau_i(y)]
\]
that is
\[
-\tau_i(y)\cdot\big(\nabla^{\Gamma_{s-r}^t}N\big)_{g_r(y)}[\tau_i(y)]=\frac{(\k_s^t)_i(y)}{1+r\,(\k_s^t)_i(y)}\,.
\]
Thus $\{\tau_i(y)\}_{i=1}^n$ is an orthonormal basis for $T_{g_r(y)}\Gamma_{s-r}^t=T_y\Gamma_s^t$ made up of eigenvalues of $\nabla^{\Gamma_{s-r}^t}N(g_r(y))$, and the last formula is just \eqref{formula curvatures}.

\bigskip

\noindent {\it Step four}: We prove that
\begin{equation}
  \label{omegasta meno zetazeta}
  |\Om^\star\setminus\zeta(Z)|=0\,.
\end{equation}
By the coarea formula \eqref{coarea} and by \eqref{cool}
\begin{eqnarray*}
|\Om^\star\setminus\zeta(Z)|&=&\int_0^\infty\H^n\Big((\Om^\star\setminus\zeta(Z))\cap\pa\Om_s\Big)\,ds
=\int_0^\infty\H^n\Big((\Om^\star\setminus\zeta(Z))\cap\Gamma_s^+\Big)\,ds
\\
&=&\int_0^\infty\H^n\Big(\Gamma_s^+\setminus\zeta(Z)\Big)\,ds\,.
\end{eqnarray*}
Since $x\in\pa^*\Om$ and $y\in\Gamma_s^+$ are such that $y=x-s\,\nu_\Om(x)$ if and only if $x=y-s\,N(y)=g_s(y)$, with $g_s$ as in step three, we have that
\[
\zeta(Z)\cap\Gamma_s^+=g_s^{-1}(\pa^*\Om)\,,\qquad\forall s>0\,.
\]
Taking into account that $\pa\Om\setminus\pa^*\Om=\Sigma$ (recall Lemma \ref{lemma prep}) and that $g_s^{-1}(\pa\Om)\subset\Gamma_s^+$, in order to prove \eqref{omegasta meno zetazeta} we are left to show that for a.e. $s>0$
\begin{equation}
  \label{reformulated}
  \H^n\big(g_s^{-1}(\Sigma)\big)=0\,.
\end{equation}
In other words, the points in $\Gamma_s^+$ that, projected over $\pa\Om$, end up on the singular set, have negligible $\H^n$-measure. We are actually going to show that \eqref{reformulated} holds for every $s>0$ such that $\H^n(\Gamma_s^+)=\H^n(\pa\Om_s)$. We shall argue by contradiction, assuming that $\H^n(\Gamma_s^+)=\H^n(\pa\Om_s)$ and
\[
\H^n\big(g_s^{-1}(\Sigma)\big)>0\,.
\]
In particular, there exists $t>s$, such that $\H^n(\Gamma_s^t\cap g_s^{-1}(\Sigma))>0$.

As a preliminary step to derive a contradiction we first notice that
\begin{equation}\label{h00}
\H^0(g_s^{-1}(x))\le 2\qquad\forall x\in\pa\Om\,.
\end{equation}
Otherwise, $g_s^{-1}(x)$ would contain at least two points $y_1$ and $y_2$ such that $(x-y_1)/|x-y_1|$ and $(x-y_2)/|x-y_2|$ are not antipodal. Any blow-up of $\var(\pa\Om,x)$ would then be a stationary varifold contained in the intersection of two non-opposite half-spaces, a contradiction to Lemma \ref{lemma cone stationary varifold}. By \eqref{h00} and by $\H^n(\Sigma)=0$ (recall \eqref{omegas meno omegas}) we find that
\[
0=2\,\H^n(\Sigma)\ge\int_\Sigma\H^0(g_s^{-1}(x))\,d\H^n=\int_{g_s^{-1}(\Sigma)}J^{\Gamma_s^t}g_s\,d\H^n\,,
\]
where
\[
J^{\Gamma_s^t}g_s=\prod_{i=1}^n(1+s(\k_s^t)_i)\ge0\qquad\mbox{on $\Gamma_s^t$}
\]
thanks to $-1/s\le(\k_s^t)_i$, see \eqref{kst}. Having assumed $\H^n(g_s^{-1}(\Sigma))>0$, and since $\{(\k_s^t)_i\}_i$ are ordered increasingly on $i$, we deduce in particular that
\begin{equation}
  \label{imply}
\H^n\Big(\Big\{y\in\Gamma_s^t:(\k_s^t)_{1}(y)=-\frac1s\Big\}\Big)\ge\H^n(\Gamma_s^t\cap g_s^{-1}(\Sigma))>0\,.
\end{equation}
By \eqref{formula curvatures} we see that
\[
  \Big\{\tilde{y}\in\Gamma_{s-r}^t:(\k_{s-r}^t)_{1}(\tilde{y})=-\frac1{s-r}\Big\}=g_r\Big(\Big\{y\in\Gamma_s^t:(\k_s^t)_{1}(y)=-\frac1s\Big\}\Big)\,.
\]
Since $g_r:\Gamma_s^t\to\Gamma_{s-r}^t$ is injective, by the area formula
\begin{eqnarray*}
\H^n\Big(\Big\{\tilde{y}\in\Gamma_{s-r}^t:(\k_{s-r}^t)_{1}(\tilde{y})=-\frac1{s-r}\Big\}\Big)
=\int_{\{y\in\Gamma_s^t:(\k_s^t)_{1}(y)=-1/s\}}J^{\Gamma_s^t}g_r\,d\H^n\,.
\end{eqnarray*}
Using again that $(\k_s^t)_i\ge-1/s$ on $\Gamma_s^t$, we have
\[
J^{\Gamma_s^t}g_r=\prod_{i=1}^n(1+r\,(\k_s^t)_i)\ge\Big(1-\frac{r}s\Big)^n>0\qquad\forall r\in(0,s)\,,
\]
so that \eqref{imply} implies that for every $r\in(0,s)$
\begin{equation}
  \label{implies}
  \H^n(\Lambda_{s-r}^t)>0\qquad\mbox{for}\,\,\Lambda_{s-r}^t=\Big\{\tilde{y}\in\Gamma_{s-r}^t:(\k_{s-r}^t)_{1}(\tilde{y})=-\frac1{s-r}\Big\}\,.
\end{equation}
By using \eqref{formula curvatures} and the fact that $a\mapsto a/(1+r\,a)$ is increasing on $a\ge0$, we see that for every $\tilde{y}\in\Lambda_{s-r}^t$, $\tilde{y}=g_r(y)$, we have
\begin{eqnarray}\nonumber
\sum_{i=1}^n(\k_{s-r}^t)_i(\tilde{y})&=&-\frac1{s-r}+\sum_{i=2}^n\frac{(\k_{s}^t)_i(y)}{1+r\,(\k_{s}^t)_i(y)}
\\\label{hey}
&\le&-\frac1{s-r}+(n-1)\frac{1/(t-s)}{1+(r/(t-s))}\le 0\,,
\end{eqnarray}
provided $r\in(r_0,s)$ for $r_0=r_0(s,t)$ suitably close to $s$, depending on $s$ and $t$. Here the choice of $0$ on the right-hand side of \eqref{hey} is arbitrary. Any constant strictly less than $n$ would suffice for the rest of the argument.

Now consider the set
\[
\Lambda=\bigcup_{r_0<r<s}\Lambda_{s-r}^t
\]
so that by the coarea formula and \eqref{implies}
\[
|\Lambda|=\int_{r_0}^s\,\H^n\big(\Lambda\cap\pa\Om_{s-r}\big)\,dr=\int_{r_0}^s\,\H^n\big(\Lambda_{s-r}^t\big)\,dr>0
\]
By the a.e. second order differentiability of $u$, there exists $y_0\in \Lambda$ such that $u$ admits a second order Taylor expansion at $y_0$. Moreover there exists $r\in(r_0,s)$ such that $y_0\in\Lambda_{s-r}^t\subset\Gamma_{s-r}^t$, so that $\nabla^2u(y_0)[N(y_0)]=0$ by \eqref{important2}, and thus
\begin{equation}
  \label{expression of u}
  \nabla^2u(y_0)=\nabla^{\Gamma_{s-r}^t}N(y_0)=-\sum_{i=1}^n(\k_{s-r}^t)_i(y_0)\,\tau_i(y_0)\otimes \tau_i(y_0)\,,
\end{equation}
thanks to \eqref{kst}. Moreover, by \eqref{hey}, we definitely have
\begin{equation}
  \label{star}
  \sum_{i=1}^n(\k_{s-r}^t)_i(y_0)\le 0\,.
\end{equation}
Let us now set $\nu=-N(y_0)$ and
\[
\D_\rho=\big\{z\in\nu^\perp:|z|<\rho\big\}\qquad\C_\rho=\big\{z+h\,\nu:z\in\D_\rho\,,|h|<\rho\big\}\qquad\rho>0\,.
\]
For every $\e>0$, the second order differentiability of $u$ at $y_0$, \eqref{star} and \eqref{expression of u} imply the existence of $\rho>0$ and of a second order polynomial $\eta:\nu^\perp\equiv\R^n\to\R$ such that $\eta(0)=0$, $\nabla\eta(0)=0$,
\begin{equation}
  \label{xi prop}
  -\Div\Big(\frac{\nabla\eta}{\sqrt{1+|\nabla\eta|^2}}\Big)(z)
  \le
  -\Div\Big(\frac{\nabla\eta}{\sqrt{1+|\nabla\eta|^2}}\Big)(0)+\e\le\sum_{i=1}^n(\k_{s-r}^t)_i(y_0)+2\e\le2\e\,,
\end{equation}
for every $z\in\D_\rho$ and
\begin{equation}
  \label{om'epigraph}
  y_0+\Big\{z+h\,\nu:z\in\D_\rho\,,-\rho<h<\eta(z)\big\}\subset (y_0+\C_\rho)\cap\Om_{s-r}\,.
\end{equation}
If we translate $\Om$ by $(s-r)\,N(y_0)$, then
\[
\Om_{s-r}\subset \Big(\Om+(s-r)\,N(y_0)\Big)
\quad\mbox{with}\quad y_0\in\pa\Om_{s-r}\cap\pa\Big(\Om+(s-r)\,N(y_0)\Big).
\]
We are now in the position to apply Theorem \ref{thm mp} with
\[
M=\pa\Big(\Om+(s-r)\,N(y_0)-y_0\Big)\,,
\]
$\nu=-N(y_0)$, $U=\D_\rho$, $z_0=0$, $h_0=\nu\cdot y_0-\rho$ and $\eta$ as in \eqref{xi prop}. Indeed by \eqref{om'epigraph} we have that if we set
\[
\vphi(z)=\inf\Big\{h\in(h_0,\infty):z+h\,\nu\in M\Big\}\qquad z\in\D_\rho\,,
\]
then $\infty>\vphi\ge\eta>h_0$ on $\D_\rho$ as well as $\vphi(0)=\eta(0)=0$. However, by \eqref{xi prop},
\[
2\e\ge-\Div\Big(\frac{\nabla\eta}{\sqrt{1+|\nabla\eta|^2}}\Big)(z)\qquad\forall z\in\D_\rho
\]
while by the constant mean curvature condition $n=H_\Om^0=\HH_{\pa\Om}\cdot\nu_\Om$ on $\pa^*\Om$ we have
\[
n=\HH_M(z+\vphi(z)\nu)\cdot\,\frac{-\nabla \vphi(z)+\nu}{\sqrt{1+|\nabla \vphi(z)|^2}}\qquad\mbox{for a.e. $z\in\D_\rho$}\,.
\]
This is a contradiction to Theorem \ref{thm mp}, hence we obtain \eqref{omegasta meno zetazeta}.

\bigskip

\noindent {\it Conclusion of the proof}: Having proved \eqref{omegasta meno zetazeta}, we can now apply the Montiel-Ros argument.  By \eqref{omega omega star 0} and \eqref{omegasta meno zetazeta},
\[
|\Om|=|\Om^\star|\le|\zeta(Z)|\le\int_Z\,\H^0(\zeta^{-1}(y))\,dy=\int_{\pa^*\Om}d\H^n_x\int_0^{1/\k_n(x)}\prod_{i=1}^n(1-t\k_i(x))\,dt\,,
\]
where $Z=\{(x,t)\in\pa^*\Om\times\R:0<t\le1/\k_n(x)\}$ and $\zeta(x,t)=x-t\,\nu_\Om(x)$. Here we have used the fact that $Z$ is a locally $\H^{n-1}$-rectifiable set in $\R^{n+1}\times\R$ with
 \begin{equation}
    \label{anis federer}
      \H^{n+1}\llcorner\Big((\pa^*\Om)\times\R\Big)=\Big(\H^n\llcorner\pa^*\Om\Big)\times\H^1\,,
 \end{equation}
 see \cite[Exercise 18.10]{maggiBOOK}, and that $J^Z\zeta=\prod_{i=1}^n(1-t\,\k_i)$. By the arithmetic-geometric mean inequality and by $\k_n\ge H_\Om^0/n$, arguing as in \eqref{same as} we thus find
\begin{eqnarray*}
\int_{\pa^*\Om}d\H^n_x\int_0^{1/\k_n(x)}\prod_{i=1}^n(1-t\k_i(x))\,dt
&\le&\int_{\pa\Om}d\H^n_x\int_0^{1/\k_n(x)}\Big(\frac1n\sum_{i=1}^n(1-t\,\k_i(x)\Big)^n\,dt
\\\nonumber
&\le&\int_{\pa\Om}d\H^n_x\int_0^{n/H_\Om^0}\Big(1-t\,H_\Om^0\Big)^n\,dt
\\
&=&\frac{n}{n+1}\,\int_{\pa\Om}\frac{d\H^n}{H_\Om^0}=|\Om|\,,
\end{eqnarray*}
so that equalities hold everywhere and
\begin{eqnarray}
    \label{condition1}
    \Big|\zeta({Z})\setminus\Om\Big|=0\,,&&
    \\
    \label{condition2}
    \H^0(\zeta^{-1}(y))=1\,,&&\qquad\mbox{for a.e. $y\in\Om$}\,,
    \\
    \label{condition3}
    \k_i(x)=\frac{H_\Om^0}n\,,&&\qquad\mbox{for every $x\in\pa^*\Om$, $i=1,...,n$}\,.
\end{eqnarray}
Recall that we have rescaled $\Om$ so that $H^0_\Om=n$. By \eqref{condition3}, since $\pa^*\Om$ is relatively open in $\pa\Om$, we can find a family $\{S_i\}_{i\in I}$, $I\subset\N$, of mutually disjoint subsets of $\pa^*\Om$ with $S_i\subset\pa B_1(x_i)$ for points $x_i\in\R^{n+1}$ such that
  \begin{equation}
    \label{partition omega star}
      \pa^*\Om=\bigcup_{i\in I}S_i\,,\qquad\mbox{$S_i$ is relatively open in $\pa\Om$}\,,\qquad\mbox{$S_i$ is connected}\,.
  \end{equation}
  Because $S_i\subset\pa\Omega$, we know that  $u(x_i)\le 1$.

  We claim that $u(x_i)=1$ for every $i\in I$. Indeed if $\de>0$ and $i\in I$ are such that $u(x_i)=1-4\delta$, then $B_\delta(x_i)\cap A_i\subset \Om$, where $A_i=\zeta(S_i\times(0,1))$ is an open subset of $\Om$. For any $y\in B_\delta(x_i)\cap A_i$, the triangle inequality implies  $u(y)<1-3\delta$, while clearly $d(y,S_i)\ge d(y,\partial B_1(x_i))\ge 1-\delta$. In particular, if $x\in\pa\Om$ is such that $|x-y|=u(y)$, then $x\not\in S_i$. Since \eqref{omega omega star 0} and \eqref{omegasta meno zetazeta} imply that for a.e. $y\in\Om$ there exist $x\in\pa^*\Om$ such that $|x-y|=u(y)$, we conclude from \eqref{partition omega star} that for a.e. $y\in B_\de(x_i)\cap A_i$ there exist $j\ne i$ and $x\in S_j$ such that $|x-y|=u(y)$; in particular, $B_\de(x_i)\cap A_i\cap A_j$ is non-empty, and since it is an open set, we have
  \[
  0<|B_\de(x_i)\cap A_i\cap A_j|\qquad\mbox{where, if $i\ne j$,}\,\, A_i\cap A_j\subset\big\{y\in\Om:\H^0(\zeta^{-1}(y))\ge 2\big\}\,.
  \]
  This is a contradiction to \eqref{condition2}. Thus $u(x_i)=1$ for every $i\in I$.

  Now let $T_i$ denote the closure of $S_i$ in $\pa B_1(x_i)$. Since $u(x_i)=1$ for every $i\in I$, we can apply Theorem \ref{thm mp} to $M=\pa\Om$ at each $x\in T_i$ to find $\rho_x>0$ such that $\pa\Om\cap B_{\rho_x}(x)=\pa B_1(x_i)\cap B_{\rho_x}(x)$. This in turn proves that $T_i=\pa B_1(x_i)$, and thus that $\pa B_1(x_i)\subset\pa\Om$ for every $i\in I$.

  Since $\H^n(\pa B_1(x)\cap\pa B_1(y))=0$ unless $x=y$, $P(\Om)<\infty$ implies that $I$ is finite. Since $\pa^*\Om$ is covered by the $S_i$, $\Om$ is the finite union of the balls $B_1(x_i)$, and owing to $\pa B_1(x_i)\subset\pa\Om$, this balls must be disjoint (their closures can of course intersect). This completes the proof of Theorem \ref{thm main}.
\end{proof}

\begin{proof}[Proof of Corollary \ref{corollary main}]
  Condition \eqref{cor hp 1} implies that the vector valued Radon measures
  \[
  \mu_{\Om_j}=\nu_{\Om_j}\,\H^n\llcorner\pa^*\Om_j
  \]
  converge in weak-star sense to $\mu_\Om$ with $|\mu_{\Om_j}|\weakstar|\mu_\Om|$ on $\R^{n+1}$. By Reshetnyak continuity theorem \cite[Theorem 2.39]{AFP}
  \[
  \lim_{j\to\infty}\int_{\R^{n+1}}\,\Phi\Big(x,\frac{d\,\mu_{\Om_j}}{d|\mu_{\Om_j}|}(x)\Big)\,d|\mu_{\Om_j}|
  =\int_{\R^{n+1}}\,\Phi\Big(x,\frac{d\,\mu_{\Om}}{d|\mu_{\Om}|}(x)\Big)\,d|\mu_{\Om}|
  \]
  whenever $\Phi\in C^0_c(\R^{n+1}\times\SS^n)$. Given $X\in C^1_c(\R^{n+1};\R^{n+1})$,
  \[
  \Phi(x,\nu)=\Div X(x)-\nu\cdot\nabla X(x)[\nu]\qquad (x,\nu)\in\R^{n+1}\times\SS^n
  \]
  belongs to $C^0_c(\R^{n+1}\times\SS^n)$ and thus we find
  \[
  \lim_{j\to\infty}\int_{\pa^*\Om_j}\Div^{\pa^*\Om_j}X\,d\H^n=\int_{\pa^*\Om}\Div^{\pa^*\Om}X\,d\H^n\,.
  \]
  By \eqref{cor hp 2} and by $\mu_{\Om_j}\weakstar\mu_\Om$
  \[
   \lim_{j\to\infty}\int_{\pa^*\Om_j}\Div^{\pa^*\Om_j}X\,d\H^n=\l\lim_{j\to\infty}\int_{\pa^*\Om_j}X\cdot\nu_{\Om_j}\,d\H^n=
   \l\,\int_{\pa^*\Om}X\cdot\nu_\Om\,d\H^n\,.
  \]
  We have thus proved that $\Om$ is a set of finite perimeter, finite volume and constant distributional mean curvature. We conclude by Theorem \ref{thm main}.
\end{proof}

\section{The Heintze-Karcher inequality for sets of finite perimeter}\label{section hk for sofp} The proof of Theorem \ref{thm main} also shows that the Heintze-Karcher inequality can be generalized to sets of finite perimeter. In this section we explain how this is done. As usual, set $u(y)=\dist(y,\pa\Om)$ for $y\in\Om$.

\begin{lemma}
  \label{lemma sofps}
  If $\Om$ is an open set with finite perimeter and finite volume in $\R^{n+1}$, then $\Om_s=\{y\in\Om:u(y)>s\}$ is an open set of finite perimeter with $\H^n(\pa\Om_s\setminus\Gamma_s^+)=0$ for a.e. $s>0$, where $\Gamma_s^+=\bigcup_{t>0}\Gamma_s^t$ and $\Gamma_s^t$ is defined as in \eqref{gammst}. Moreover:

  \medskip

  \noindent (i) For every $s>0$, $\Gamma_s^+$ can be covered by countably many graphs of $C^{1,1}$-functions from $\R^n$ to $\R^{n+1}$.

  \medskip

  \noindent (ii)  For every $s>0$, the principal curvatures $(\k_s)_i$ of $\Gamma_s^+$ are defined $\H^n$-a.e. on $\Gamma_s^+$ by setting
  \begin{eqnarray*}
  (\k_s)_i=(\k_s^t)_i\qquad\mbox{on $\Gamma_s^t$ for each $t>s$}\,,
  \end{eqnarray*}
  for $(\k_s^t)_i$ as in \eqref{kst}. Correspondingly, $\H^n$-a.e. on $\Gamma_s^+$ we can define
  \[
  H_{\Om_s}=\sum_{i=1}^n(\k_s)_i \qquad |A_{\Om_s}|^2=\sum_{i=1}^n\,(\k_s)_i^2
  \]
  as natural generalizations of the mean curvature and of the length of the second fundamental form of $\pa\Omega_s$ with respect to $\nu_{\Om_s}$ at points in $\Gamma_s^+\subset\pa\Om_s$.

  \medskip

  \noindent (iii) For every $r<s<t$, the map $g_r:\Gamma_s^t\to\Gamma_{s-r}^t$, defined by $g(y)=y-r\,\nabla u(y)$ for $y\in\Gamma_s^t$, is a Lipschitz bijection from $\Gamma_s^t$ to $\Gamma_{s-r}^t$, with
  \begin{equation}
    \label{midterms}
      J^{\Gamma_s^t}g_r(y)=\prod_{i=1}^n\big(1+r(\k_s)_i(y)\big)\qquad (\k_{s-r})_i(g_r(y))=\frac{(\k_s)_i(y)}{1+r\,(\k_s)_i(y)}
  \end{equation}
  for $\H^n$-a.e. $y\in\Gamma_s^t$.
\end{lemma}

\begin{proof}
  All these conclusions are contained in step one, two and three of the proof of Theorem \ref{thm main}, where at no stage the constant distributional mean curvature condition, or the regularity of $\pa^*\Om$ implied by it, have been used.
\end{proof}

  As a consequence of Lemma \ref{lemma sofps}, we see that for every $x\in g_s(\Gamma_s^+)\subset\pa\Om$, the limit
  \begin{equation}
    \label{ki}
    \k_i(x)=\lim_{r\to s^-}(\k_{s-r})_i(x)\in[-\infty,\infty)
  \end{equation}
  exists by monotonicity, see \eqref{midterms}. We thus give the following definitions: given an open set of finite perimeter and finite volume $\Om\subset\R^{n+1}$ we define the {\it viscosity boundary of $\Om$} as
  \[
  \pa^v\Om=\bigcup_{s>0} g_s(\Gamma_s^+)
  \]
  and the {\it viscosity mean curvature of $\Om$} by
  \begin{equation}
    \label{homega visco}
  H_\Om^v(x)=\sum_{i=1}^n\k_i(x)\qquad\forall x\in\pa^v\Om\,.
  \end{equation}
  Notice that $\pa^v\Om$ is covered by countably many $\H^n$-rectifiable sets, although it may contain points of $\spt\mu_\Om$ that are outside the reduced boundary, or that have density $1$ for $\Om$. It is not obvious if, at this level of generality, $\pa^v\Om$ is $\H^n$-finite. In any case, our only reason for introducing these concepts is to formulate the following definition: a set of finite perimeter and finite volume $\Om$ is {\it mean convex in the viscosity sense} if $H_\Om^v$ defined in \eqref{homega visco} is positive along $\pa^v\Om$. It is easy to see that if $\partial\Omega$ is $C^2$, then $\pa^v\Om=\pa\Om$ and $H_\Om^v(x)=H_\Om(x)$ for any $x\in \pa\Om$. Hence, the viscosity notion generalizes the mean convexity in the classical sense.

  This said, following Brendle's point of view on the Montiel-Ros argument \cite{brendle}, we have the following generalized form of the Heintze-Karcher inequality, see \eqref{hk visco} below.

  \begin{theorem}[Heintze-Karcher inequality for sets of finite perimeter]
    \label{thm viscous hk}
    If $\Om\subset\R^{n+1}$ is an open set of finite perimeter and finite volume which is mean convex in the viscosity sense, then for every $s>0$
    \begin{equation}
      \label{hk visco}
      |\Om_s|\le\frac{n}{n+1}\int_{\Gamma_s^+}\frac{d\H^n}{H_{\Om_s}}\,.
    \end{equation}
    Moreover, the limit of the right-hand side of \eqref{hk visco} as $s\to 0^+$ always exists in $(0,\infty]$.
  \end{theorem}

  \begin{proof} The mean convexity assumption on $\Om$ and the monotonicity property behind the definition \eqref{ki} of $\k_i$, imply that $\sum_{i=1}^n(\k_s)_i>0$ on $\Gamma_s^+$. We define for every $s>0$
  \[
  Q(s)=\int_{\Gamma_s^+}\,\frac{d\H^n}{H_{\Om_s}}>0\,.
  \]
  Moreover, for every $t>0$ we define $Q^t:(0,t)\to(0,\infty)$ by setting
  \[
  Q^t(s)=\int_{\Gamma_s^t}\frac{d\H^n}{H_{\Om_s}}\,,\qquad s\in(0,t)\,.
  \]
  Notice that
  \begin{equation}
    \label{notice thatt}
      Q(s)\ge Q^t(s)\ge Q^{t+\e}(s)\qquad\forall t>s\,,\e>0\,,
  \end{equation}
  and recall that $\H^n(\Gamma_s^t)$ converges monotonically to $\H^n(\Gamma_s^+)$ as $t\to s^+$, so that
  \begin{equation}
    \label{qts approx}
      Q(s)=\lim_{t\to s^+}Q^t(s)=\sup_{t>s}Q^t(s)\,,\qquad\mbox{for every $s>0$}\,.
  \end{equation}
  For $r\in(0,s)$ by Lemma \ref{lemma sofps}--(iii) we have
  \begin{eqnarray*}
  Q^t(s-r)-Q^t(s)&=&
  \int_{\Gamma_s^t}\Big(\frac{\prod_{i=1}^n(1+r\,(\k_s)_i)}{\sum_{i=1}^n(\k_s)_i/(1+r(\k_s)_i)}-\frac1{H_{\Om_s}}\Big)\,d\H^n
  \\
  &=&
  \int_{\Gamma_s^t}\Big(\frac{1+r\,H_{\Om_s}+O_t(r^2)}{H_{\Om_s}-r\,|A_{\Om_s}|^2+O_t(r^2)}-\frac1{H_{\Om_s}}\Big)\,d\H^n
  \end{eqnarray*}
  where $O_t(r^2)/r\to 0$ uniformly on $\Gamma_s^t$ as $r\to 0$. We thus find that $Q^t$ is differentiable on $(0,t)$ with
  \[
  (Q^t)'(s)=-\int_{\Gamma_s^t}1+\frac{|A_{\Om_s}|^2}{H_{\Om_s}^2}\,d\H^n\,,\qquad\forall s\in(0,t)\,.
  \]
  By the Cauchy-Schwartz inequality, $H_{\Om_s}^2\le n\,|A_{\Om_s}|^2$. Hence,
  \begin{equation}
    \label{integ}
      (Q^t)'(s)\le -\frac{n+1}n\,\H^n(\Gamma_s^t)\,,\qquad\forall s\in(0,t)\,.
  \end{equation}
  If $0<s_1<s_2$, then by \eqref{qts approx}, \eqref{notice thatt} and \eqref{integ} respectively, we have
  \begin{eqnarray}\nonumber
    Q(s_1)-Q(s_2)&=&\lim_{\e\to0^+}Q^{s_1+\e}(s_1)-Q^{s_2+\e}(s_2)
    \\\nonumber
    &\ge&    \lim_{\e\to0^+}Q^{s_2+\e}(s_1)-Q^{s_2+\e}(s_2)=Q^{s_2}(s_1)-Q^{s_2}(s_2)
    \\\label{coming back}
    &\ge&\frac{n+1}n\,\int_{s_1}^{s_2}\H^n(\Gamma_s^{s_2})\,ds\,,
  \end{eqnarray}
  and, in particular, $Q$ is decreasing on $(0,\infty)$. Again by Lemma \ref{lemma sofps}--(iii)
  \[
  \H^n(\Gamma^t_{s-r})=\int_{\Gamma^t_s}\prod_{i=1}^n(1+r(\k_i)_s)\,d\H^n
  \]
  where $1+r(\k_i)_s\to 1$ uniformly on $\Gamma_s^t$ as $r\to 0$ thanks to $1/(t-s)\ge(\k_s)_i\ge-1/s$ for every $i=1,...,n$. Thus $\H^n(\Gamma_s^t)$ is continuous on $s\in(0,t)$, and
  \[
  \int_{s_1}^{s_2}\H^n(\Gamma_s^{s_2})\,ds=(s_2-s_1)\,\H^n(\Gamma_{s^*}^{s_2})
  \]
  for a suitable $s_*\in(s_1,s_2)$. But \eqref{proj pp} implies
  \[
  \liminf_{s\to (s_2)^-}\H^n(\Gamma_s^{s_2})\ge\H^n(\pa\Om_{s_2})\,,
  \]
  so that, in conclusion,
  \[
  \liminf_{s_1\to (s_2)^-}\frac1{s_2-s_1}\int_{s_1}^{s_2}\H^n(\Gamma_s^{s_2})\,ds\ge\H^n(\pa\Om_{s_2})\,,\qquad\forall s_2>0\,.
  \]
  Coming back to \eqref{coming back}, and noticing that $Q'(s)$ exists for a.e. $s>0$ by monotonicity, we conclude that
  \[
  -Q'(s)\ge\frac{n+1}n\,\H^n(\pa\Om_s)\qquad\mbox{for a.e. $s>0$}\,.
  \]
  We integrate this inequality over $(s,\infty)$ to complete the proof of \eqref{hk visco}.
\end{proof}

\bibliography{references}
\bibliographystyle{is-alpha}

\end{document}